\documentclass{mcom-l}
\usepackage{amssymb,amscd,float}
\usepackage{algorithmic,algorithm}
\usepackage[pdftex]{graphicx}
\usepackage{url}
\usepackage{tikz-cd}

\usepackage{epsf}
\usepackage{graphicx}% Include figure files
\usepackage{psfig}% Include figure files
\usepackage{dcolumn}% Align table columns on decimal point
\usepackage{bm}% bold math
\usepackage{subfigure}
\usepackage{color}
\usepackage{amsthm}
\usepackage{amssymb}
\usepackage{amsmath}
\usepackage{pb-diagram}
\usepackage{algorithm}
\usepackage{algorithmic}

\newtheorem{theorem}{Theorem}[section]
\newtheorem{lemma}[theorem]{Lemma}

\theoremstyle{definition}
\newtheorem{definition}[theorem]{Definition}

\newtheorem{corollary}[theorem]{Corollary}

\theoremstyle{remark}
\newtheorem{remark}[theorem]{Remark}

\numberwithin{equation}{section}

\newcommand{\abs}[1]{\lvert#1\rvert}

\newcommand{\N}{{\mathbb{N}}}
\newcommand{\R}{{\mathbb{R}}}

\newcommand{\Z}{{\mathbb{Z}}}

\newcommand{\bbf}{{\bf f}}
\newcommand{\bx}{{\bf x}}
\newcommand{\by}{{\bf y}}

\newcommand{\bs}{{\bf s}}

\newcommand{\cB}{{\mathcal B}}

\newcommand{\cE}{{\mathcal E}}

\newcommand{\cG}{{\mathcal G}}
\newcommand{\cH}{{\mathcal H}}

\newcommand{\cN}{{\mathcal N}}

\newcommand{\cX}{{\mathcal X}}

\newcommand{\setof}[1]{\left\{ {#1}\right\}}

\newcommand{\pd}{{\mathsf{PD}}}

\newcommand{\cl}{ \mbox{\rm{cl}}}
\newcommand{\norm}[1]{\left\Vert#1\right\Vert}
\newcommand{\im}{\mbox{\rm{im}}}
\newcommand{\bd}{ \mbox{\rm{bd}}}
\newcommand{\inter}{ \mbox{\rm{int}}}

\begin{document}

\title{On $\varepsilon$ Approximations of Persistence Diagrams}

\author{Jonathan Jaquette}
\address{Department of Mathematics,
Hill Center-Busch Campus,
Rutgers University,
110 Frelinghusen Rd,
Piscataway, NJ  08854-8019, USA}
%%\curraddr{Department of Mathematics,
%%%Hill Center-Busch Campus,
%%Rutgers University,
%%110 Frelinghusen Rd,
%%Piscataway, NJ  08854-8019, USA}
\email{jaquette@math.rutgers.edu}
\thanks{The first author's research was funded in part by AFOSR Grant FA9550-09-1-0148 and NSF Grant DMS-0915019.}

\author{Miroslav Kram\' ar}
\address{Department of Mathematics,
Hill Center-Busch Campus,
Rutgers University,
110 Frelinghusen Rd,
Piscataway, NJ  08854-8019, USA}
\email{miroslav@math.rutgers.edu}
\thanks{The second author's research was funded in part by NSF Grants DMS-1125174 and DMS-0835621.}

%%%%%%%%%%%%%%%%%%%%%%%%%%%%%%%%%%%%%%%%%%%%%%%%%%%%%%%%%
%%TODO
%%%%%%%%%%%%%%%%%%%%%%%%%%%%%%%%%%%%%%%%%%%%%%%%%%%%%%%%%
%    General info
%\subjclass[2000]{Primary 54C40, 14E20; Secondary 46E25, 20C20}
%\date{January 1, 2001 and, in revised form, June 22, 2001.}
%\keywords{Differential geometry, algebraic geometry}

\begin{abstract}

Biological and physical systems often exhibit distinct structures at different spatial/temporal scales. Persistent homology is an algebraic tool that provides a mathematical framework for analyzing the multi-scale structures frequently observed in nature. In this paper  a theoretical framework for the algorithmic computation of an arbitrarily good approximation of the persistent homology is developed. We study the filtrations generated by sub-level sets of a function $f \colon X \to \R$, where $X$ is a CW-complex. In the special case $X = [0,1]^N$, $N \in \N$ we discuss  implementation of the proposed algorithms. We also investigate  \emph{a priori} and  \emph{a posteriori} bounds of the approximation error   introduced by our method.

\end{abstract}

\maketitle
%%%%%%%%%%%%%%%%%%%%%%%%%%%%%%%%%%%%%%%%%%%%%%%%%%%%%%%%%%%%%%%%%%
\section{Introduction}
%%%%%%%%%%%%%%%%%%%%%%%%%%%%%%%%%%%%%%%%%%%%%%%%%%%%%%%%%%%%%%%%%%
\label{sec::Intro}

The formation of complex patterns can be encountered in almost all sciences.  Many patterns observed in  nature are extremely complicated and  it is usually impossible  to completely understand  them   using  analytical methods. Often a coarser but computationally tractable description is needed. In recent years computational topology \cite{MR2572029, KMM04computational} has become widely recognized as an important tool for quantifying complex structures.   Computational homology has been used  to quantitatively study spatio-temporal chaos \cite{PhysRevE.70.035203} as well as to  analyze complicated flow patterns generated by Rayleigh-B\'enard convection \cite{KKSGMM07, FLM:8352228}, time evolution of the isotropically  compressed granular material \cite{epl12} and  complex microstructures generated in binary metal alloys through a process called spinodal decomposition \cite{Gameiro05evolutionof, MTSV06Evolution}. 
Unfortunately homology groups can be extremely sensitive to small perturbations \cite{pD}. Persistent homology  \cite{HH08survey} overcomes this problem  \cite{MR2279866} and provides a more powerful tool for studying evolution of complex patterns in the granular media \cite{pD,pre13} and fluid dynamics \cite{RBC-KF}.

In the applications mentioned above  the pattern is given by a sub- or super-level set of a real valued function $f$. Therefore analyzing the patterns usually  involves numerical study  based on a suitable discretization. It is important to know when the homology groups of a sub-level set of $f$ can be inferred from a discrete approximation of this set. 
Probabilistic results, for a  function $ f \colon [a,b] \subset \R \to \R$, are presented in    \cite{MW07random}. 
An alternative approach is to  construct a cubical approximation of the sub-level set with the same homology groups \cite{DKW09verified}. 
The algorithm presented in \cite{DKW09verified} can handle functions defined on a unit square and uses a uniform grid on $[0,1]^2$. If the resolution of a grid is fine enough, then the sub-level set can be approximated by a collection of the grid elements that intersect the sub-level set. This often leads to an unnecessary large complex. The size of the complex can be reduced  by using a randomized version of the algorithm \cite{CWD13randomized}.
Using a regular CW complex to construct the nodal approximation \cite{DKMW10coreduction} simplifies the  homology computations.

In this paper we develop a method  for approximating the sub-level sets  of a continuous function  $f : X \to \R$, where $X$ is a regular CW complex.
Note that the sub-level set $X_t = f^{-1}(-\infty, t] = \setof{ x \in X \colon f(x) \leq t}$  does not have to be a CW complex. 
The following definition provides conditions on a CW structure $\cE$ of $X$ under which the singular homology groups of the CW complex 
\[
\cX_t = \bigcup \setof{ e \in \cE \colon f(\cl(e)) \leq t },
\]
are isomorphic to the homology groups $H_*(X_t)$. Throughout the paper, we use the notation ``$f(\cl(e)) \leq t$'' to mean ``$f(x)\leq t$ for all $x \in cl(e)  $''.

\begin{definition}
\label{def::CompatibleCW}
Let $X$ be a finite regular  complex  with a CW structure $\cE$. Suppose that  $f:X \to \R$ is  a continuous function and  $t\in\R$.  We say that the CW structure $\cE$ is \emph{compatible} with $X_t$ if   every cell $e \in \cE$  satisfies one of the following conditions:
\begin{enumerate}
\item $f(\cl(e)) > t$,
\item $f(\cl(e)) \leq t$,
\item there exists a deformation retraction  $h(\bx,s) : X_t \cap \cl(e) \times  [0,1] \to  X_t \cap \cl(e)$     of the set $X_t \cap \cl(e)$ onto $X_t \cap \bd(e)$.
\end{enumerate}
\end{definition}

According to the following theorem, proved in Section~\ref{sec::HomologyOnCW}, the singular homology groups $H_*(X_t)$  can be obtained by computing the  cellular homology groups of the complex $\cX_t$.

\begin{theorem}
\label{prop::IntroCorrect_homology}
Let $X$ be a finite regular CW complex. If the CW decomposition $\cE$ of $X$  is compatible with $X_t$, then the map $i_*: H_*(\cX_t) \to H_*(X_t) $, induced by the inclusion $i : \cX_t \to X_t$, is an isomorphism.
\end{theorem}

Due to computational complexity of computing homology, it is extremely important that the number of cells  in the complex $\cX_t$ is  as small as possible. To achieve this goal we use a multi-scale CW structure $\cE$. The diameter (size) of the cells in $\cE$ varies   with  local complexity  of the pattern.  Larger cells are used at places with smaller complexity while  small cells might be necessary at the places where  spacial scales of the pattern are small.
To produce a multi-scale CW structure we start from a coarse CW structure of $X$ and keep refining the cells that do not satisfy any of the  conditions in Definition~\ref{def::CompatibleCW}.  For $X = [0,1]^2$ the number of cells in $\cX_t$  is smaller than the number of cells in the cubical complex used in \cite{DKW09verified,DKMW10coreduction}.

To facilitate the refinement  process  we introduce  self similar grids.  The size of the grid elements might vary but all the grid elements of $\cG$ are homeomorphic  to $X$. The resolution of the self similar grid $\cG$ can be dynamically adjusted at different parts of $X$. To increase the resolution, each grid element can be refined into a union of smaller complexes homeomorphic to $X$. 
We represent a self similar grid by a tree structure. The root of this tree corresponds to  the whole  complex $X$. The $i$-th level nodes represent the complexes obtained by $i$ refinements of $X$. Finally, the leaves correspond to the grid elements of $\cG$. 

The tree representation of $\cG$ provides a basis for a memory efficient data structure for storing the complex $\cX_t$ and its boundary operator. We stress that the boundary operator is not stored as a large matrix. Instead it is computed using the boundary operator $\partial \colon H_*(X) \to H_{*-1}(X)$ corresponding to the coarse (unrefined) CW structure, the tree structure representing the self similar grid and geometric information about intersections of the grid elements.

While constructing a compatible CW structure for an arbitrary regular CW complex $X$ might be challenging, it is  straight forward  when $X = [0,1]^N$, where $N \in \N$.  The coarse CW structure of $[0,1]^N$, employed in this paper, is given by standard decomposition  of a cube into vertices, edges and higher dimensional faces (cells). A self similar grid $\cG$  on $[0,1]^N$ consists of dyadic cubes of different sizes. If every cube (grid element) in $\cG$ satisfies the following definition, then every cell in the CW structure $\cE$ generated by the grid $\cG$ satisfies one of the conditions in  Definition~\ref{def::CompatibleCW}. Moreover, if a cell $e\in \cE_{s_i}$ satisfies Condition~(3) in Definition~\ref{def::CompatibleCW}, then the function $f\circ h(\bx,s)$ can be taken to be non-increasing in $s$.

\begin{definition}
\label{def::IntroAnalyticVerified}
Let $f\in C^1([0,1]^N,\R)$ and $C$ be an $N$ dimensional dyadic cube. We say that $C$   is $(f,t)$-\emph{verified} if 
$0 \not\in \frac{\partial f}{\partial x_{l_i}}(C)$ for some coordinate vectors  $x_{l_1}, \ldots, x_{l_n}$, where $0 \leq n \leq N$,   and  for every $N-n$ dimensional cell $e$ of the cube $C$ which is orthogonal to the vectors  $x_{l_1}, \ldots, x_{l_n}$, then either $ f(\cl(e)) > t $ or $f(\cl(e)) \leq  t $.
\end{definition}

A cell $e$ is considered to be orthogonal to a collection of vectors $v_1, \dots , v_n$ if the projection of $e$ onto $\mbox{span}\{ v_1, \dots , v_n\}$ is zero dimensional. 
The condition in Definition~\ref{def::IntroAnalyticVerified} can be checked using interval arithmetics. Therefore the homology groups $H_*(X_t)$ can be evaluated on a computer.  However, if $t$ is close to a critical value, then using interval arithmetics to guarantee that the cubes surrounding the critical point satisfy Definition~\ref{def::IntroAnalyticVerified} might require an extremely fine resolution or might not be possible at all. Another problem with this approach is that the topology of the set $X_t$  depends on the  choice of the threshold  $t$, and small changes of $t$ can lead to large changes of the topology as demonstrated in \cite{pD}. So  the particular choice of the threshold $t$ might be hard to justify. Both of these problems can be overcome by using the persistent homology \cite{HH08survey, MR2572029}.

Persistent homology relates the homology groups of $X_t$ for different values of $t$.   
It reduces the function $f$ to a collection of points in the plane. This collection of points is called a {\em persistence diagram} and denoted by $\pd(f)$. Each point in the persistence diagram encodes a well defined geometric feature of $f$. Hence the persistent homology provides a finer description of the pattern than the homology groups of any one sub-level set. The set of all persistence diagrams can be turned into a complete metric space by using a {\em Bottleneck distance} $d_B$ between the diagrams~\cite{MR2279866, 0266-5611-27-12-124007}.

The main contribution of this paper is a theoretical  framework for constructing an $\varepsilon$ approximation of the persistence diagram $\pd(f)$. For  $\varepsilon >  0$ and a continuous function $f \colon X \to \R$ we define an  $\varepsilon$ approximate filtration of $f$ as follows.

\begin{definition}
Let $X$ be a finite regular CW complex and $f : X \to \R$  a continuous function.  Suppose that a sequence of real numbers $\setof{s_i}_{i=-1}^{m+1}$ has the following properties:
\begin{enumerate}
\item $s_{-1} = -\infty$ and  $s_0 <  \min_{x\in X} f(x)$, 
\item $s_{m} > \max_{x\in X} f(x)$ and $s_{m+1} = \infty$,
\item $\abs{s_{i} - s_{i-1}} < \varepsilon$ for $1 \leq i \leq m$.
\end{enumerate}
Then the collection of sub-level sets  $\setof{X_{s_i}}_{i = -1}^{m+1}$ is called an $\varepsilon$ \emph{approximate filtration} of $f$. 
\label{def::ApproximateFiltration}
\end{definition}

It follows from the stability results \cite{MR2279866} that the persistence diagram $\pd'$ of any $\varepsilon$ approximate filtration of $f$ is close to $\pd(f)$. More precisely 
\begin{equation}
\label{eqn::IntroEstimate}
d_B(\pd(f), \pd') < \varepsilon.
\end{equation}
Construction of the persistence diagram $\pd'$ requires determining homology groups of the sub-level sets in some $\varepsilon$ approximate filtration of $f$. Freedom in choosing an $\varepsilon$ approximate filtration of $f$ allows us to stay away from the critical values.  In order to compute the persistence diagram of an  $\varepsilon$ approximate filtration $\setof{X_{s_i}}_{i = -1}^{m+1}$ one might be tempted to replace this filtration by its cellular counterpart $\setof{\cX_{s_i}}_{i = -1}^{m+1}$.
However  there is no guarantee that $\cX_{s_i} \subseteq \cX_{s_{i+1}}$ and so $\setof{\cX_{s_i}}_{i=-1}^{m+1}$ may not be a filtration. (See Figure~\ref{fig::NotAFiltration} for a counterexample.)
In order to construct a cellular filtration we have to require that the CW structures $\setof{\cE_{s_i}}_{i=-1}^{m+1}$ are commensurable.

\begin{definition}
Let $\setof{\cE_{s_i}}_{i = -1}^{m+1}$ be a collection of CW structures on $X$. We say that they are {\em commensurable}  if for every  $e, e' \in \cE' := \bigcup \cE_{s_i}$ such that $e \cap e' \neq \emptyset$ then either $e\subseteq e'$ or   $e'\subseteq e$.
\label{def::commensurable}
\end{definition}

The CW structures $\setof{\cE_{s_i}}_{i = -1}^{m+1}$ constructed using a  self similar grid are always commensurable. The CW complexes  $\setof{\bar{\cX}_{s_i}}_{i =-1}^{m+1}$  define by   $\bar{\cX}_{s_i} = \bigcap_{j \geq i} \cX_{s_j}$ form an $\varepsilon$ approximate cellular filtration of $f$. This filtration has the following property.  

\begin{theorem}
\label{prop::ApproximateFiltration}
Let $\setof{\bar{\cX}_{s_i}}_{i = -1}^{m+1}$  be an $\varepsilon$ approximate cellular filtration of $f$ corresponding to  $\setof{\cE_{s_i}}_{i = -1}^{m+1}$.
Suppose that for every cell $e\in \cE_{s_i}$ which satisfies Condition~(3) in Definition~\ref{def::CompatibleCW}, the function $f\circ h(\bx,s)$ is non-increasing in $s$. Then the persistence diagrams of the filtrations $\setof{{X}_{s_i}}_{i = -1}^{m+1}$  and $\setof{\bar{\cX}_{s_i}}_{i = -1}^{m+1}$ are equal.
\end{theorem}

\begin{corollary}
Let $\setof{\bar{\cX}_{s_i}}_{i = -1}^{m+1}$  be an $\varepsilon$ approximate cellular filtration of $f$.  If $\pd'$ is the persistence diagram of this filtration, then $d_B(\pd(f), \pd') < \varepsilon$.
\end{corollary}

The persistence diagram $\pd'$ can be computed even if we fail to construct  CW structures compatible with  $X_{s_i}$ for some of the thresholds $s_i$.  In this case the distance $d_B(\pd(f), \pd')$ can be evaluated \emph{a posteriori}. The required resources for computing the persistence homology can be further reduced if  zig-zag persistence \cite{Carlsson:2009:ZPH:1542362.1542408} is used.

The remainder of this paper is organized as follows. In Section~\ref{sec::Background} we recall the definition of a CW-complex, homology and persistent homology.  Theorem~\ref{prop::IntroCorrect_homology} is proven in Section~\ref{sec::HomologyOnCW}.
In Section~\ref{sec::PD} we address the Inequality~(\ref{eqn::IntroEstimate}).
A precise definition of a self similar grid is given in Section~\ref{sec:grid}. These grids are employed in Section~\ref{sec::CompatibleCW}  to build a compatible $CW$ structure. 
A memory efficient data structure for storing the CW complex $\cX_t$  and its  boundary operator is presented  in Section~\ref{sec:DataAndAlgorithms}.
In the last section we use our method to compute the approximate persistence diagram of a two dimensional Fourier series with $5$ modes and discuss the \emph{a posteriori} error estimates.

%%%%%%%%%%%%%%%%%%%%%%%%%%%%%%%%%%%%%%%%%%%%%%%%%%%%%%%%%%%%%%%%%%
\section{Background}
%%%%%%%%%%%%%%%%%%%%%%%%%%%%%%%%%%%%%%%%%%%%%%%%%%%%%%%%%%%%%%%%%%
\label{sec::Background}

In this section  we recall definitions of a CW-complex \cite{H01algebraic}, homology \cite{KMM04computational}  and persistent homology \cite{HH08survey}. Notation introduced in this section is used throughout  the paper.  We start by introducing a regular CW-complex.

Let $X$ be a Hausdorff space. A \emph{(finite) regular CW decomposition} of $X$ is a (finite) set $\cE$ of subspaces of $X$ with the following properties:

\begin{enumerate}
\item  $\cE$ is a partition of $X$, that is $X = \bigcup_{e\in\cE} e$ and $e \neq e' \Rightarrow e \cap e' = \emptyset,$
\item  Every $e\in\cE$ is homeomorphic to some euclidean space $\R^{\dim(e)},$
\item  For every $n$-cell  $e\in\cE$ there exists a continuous embedding $\Phi_e(B^n,S^{n-1}) \to (X^{n-1} \cup e, X^{n-1})$ such that $\Phi_ e(B^n) = \cl( e)$ where $\cl( e)$ is the closure of $ e$, $B^n = \setof{x \in \R^n : \norm{x} \leq 1}$ and $S^n = \setof{x \in B^{n+1} : \norm{x} = 1}$.
\end{enumerate} 
The number $\dim(e)$ is well determined and is called the dimension of $e$. The sets $e\in\cE$ which are homeomorphic to  $\R^n$ are the $n$\emph{-cells}. The boundary of a cell $e$ is defined by $\bd(e) := \cl(e) \setminus e$. The union $X^n = \bigcup_{dim(e)\leq n}e$ is the $n$\emph{-skeleton}  of the CW decomposition.
The map $\Phi_e$ is called a characteristic map for $e$ and $\varphi_e = \Phi_e|_{S^{n-1}} :  S^{n-1}  \to X^{n-1}$ an attaching map. 
The space $X$ is called a regular CW complex if there exists a regular CW decomposition of $X$. 

Cellular homology is useful for computing the homology groups of CW complexes.  For a regular CW-complex $X$ the  $n$-chains are defined by  
\[
C_n(X) = H_n(X^n,X^{n-1}).
\]
There is a one-to-one correspondence between the basis elements of $C_n(X)$ and the $n$-cells of $X$. For every $n$ dimensional cell $e$ the map $\Phi_e$ generates a monomorphism $(\Phi_e)_* : H_n(B^n,S^{n-1}) \to C_n(X)$. For a fixed generator $a \in H_n(B^n,S^{n-1})$  the set 
\[
\Lambda^n = \{ (\Phi_e)_*(a) : e \in \cE \mbox{ and } \dim(e) = n \}
\]
is a basis  of $C_n(X)$.  
For every $\lambda \in \Lambda= \bigcup_{n\in \Z}\Lambda^n$ we denote its corresponding geometric  cell by  $|\lambda|$.

The boundary  operator  $\partial_n : C_n(X) \to C_{n-1}(X)$ is defined by the following  composition  
\[
H_n(X^n,X^{n-1}) \stackrel{d_{n}}{\rightarrow}  H_{n-1}(X^{n-1}) \stackrel{j_{n-1}}{\rightarrow} H_{n-1}(X^{n-1},X^{n-2})
\]
where $d_n$ is the connecting homomorphism from the long exact sequence of the pair $(X^n , X^{n-1})$ and $j_{n-1}$ is the quotient map. 
In the above chosen  basis of $C_n(X)$, the boundary operator $\partial_n : C_n(X) \to C_{n-1}(X)$ is completely determined by the values  on the basis elements $\Lambda^n$ and can be uniquely expressed as 
$$
\partial_n( \lambda ) = \sum_{ \mu \in \Lambda^{n-1} }[  \lambda :  \mu]   \mu.
$$
The  coefficient $[ \lambda :  \mu ]$ is called an incidence number of the cells.  Regularity of the complex $X$ implies that the incidence number is non-zero only for the pairs of cells   $\lambda \in \Lambda^n, \tau\in\Lambda^{n-1}$  such that 
 $ \bd(|\lambda|) \cap |\tau|  \neq \emptyset$.  Moreover   $[  \lambda : \tau ] = \pm1$ in this case.

The cellular complex corresponding to $X$ is defined by $C(X) = \setof{C_n(X),\partial_n}_{n\in \Z}$.  
The singular homology $H_*(X)$ of $X$ is isomorphic to the cellular homology of $C(X)$, which is given by 
\[
H_*(C(X)) = \frac{\ker(\partial_* (C_*(X)) )}{\im(\partial_{*+1} (C_{*+1}(X)))}.
\]

We close this section with a short overview of persistent homology,  following the presentation given in \cite{HH08survey}.
Let $f : X \to \R$ be a function defined on a CW complex $X$. 
We denote its  sub-level sets by  $X_t = f^{-1}(-\infty, t]$. The function $f$ is called {\em tame} if the homology groups of every sub-level set have finite ranks and there are only finitely many values $t$ across which the homology groups are not isomorphic. 
Let $t_1  < t_2 \ldots < t_m$ be the values for which the homology of $X_{t_{i}- \epsilon}$ and $X_{t_{i}+ \epsilon}$ differ for all $0 < \epsilon <<1$. 
We define an interleaved sequence $\setof{s_i}_{i=-1}^{m+1}$ such that $s_{-1} =  -\infty$,  $s_{m+1}  = \infty$ and $s_{i-1} < t_i < s_i$, for $1\leq i \leq m$. For each $-1 \leq i \leq j \leq m+1$ there is a natural inclusion of the set $X_{s_i}$ into $X_{s_j}$ and we denote the induced homomorphism between the corresponding homology groups by
\[
\bbf^{i,j}_n : H_n(X_{s_i}) \to H_n(X_{s_j}).
\]
Persistent homology makes use of the above mentioned
maps to compare the homology groups of $X_t$ for different values of $t$. 
We say that the homology class  $\alpha \in H_n( X_{s_i})$ is {\em born} at $t_i$ if it does not come from a class in $ H_n(X_{s_{i-1}})$, that is $\alpha \not\in \im ( \bbf^{i-1,i}_n )$. The class $\alpha$ {\em dies} at $t_j$ if $\alpha$ is in  $\im(\bbf^{i,j-1}_n)$ but not in $\im(\bbf^{i,j}_n)$. We denote the birth-death pair for $\alpha$ by $(b_\alpha, d_\alpha)$.   
The $n$-th persistence diagram  $\pd_n(f)$ is defined to be the multi set of the points containing:
\begin{enumerate}
\item one point for each pair  $(b_\alpha, d_\alpha)$.
\item infinitely many copies of the points $(s,s)$ at the diagonal.
\end{enumerate}
If $X$ has dimension $n$, then the persistence diagram for the function $  f : X \to \R$ is the collection $\pd(f) = \setof{\pd_i(f)}_{i = 0 }^n$.  Actually, persistence diagrams can be defined for continuous function $f \colon X \to \R$, where $X$ is the realization of a simplicial complex as a topological space\cite{S-P-Modules}. In this case the persistence diagrams can contain accumulation points at the diagonal. In other words there might be a sequence of off diagonal points converging to the diagonal. Using the same reasoning as in \cite{S-P-Modules}  the same is true if $X$ is  a finite regular CW-complex.

Differences between the persistence diagrams can be assessed using a bottleneck distance defined as follows.

\begin{definition}
\label{defn:bottleneck}
Let $\pd =  \setof{\pd_i}_{i=0}^n$ and $PD' = \setof{\pd'_i}_{i=0}^n$ be persistence diagrams. 
The {\em bottleneck distance} between $\pd$ and $\pd'$ is defined to be 
\[
d_B(\pd,\pd') = \max_{i}{\inf_{\gamma\colon \pd_i \to \pd'_i} \sup_{p\in\pd_i} \| p -\gamma(p)\|_\infty }, 
\]
where $\|(a_0,b_0)-(a_1,b_1)\|_\infty := \max\setof{|a_0-a_1| , |b_0-b_1|}$ and $\gamma$  ranges over all bijections.
\end{definition}

%%%%%%%%%%%%%%%%%%%%%%%%%%%%%%%%%%%%%%%%%%%%%%%%%%%%%%%%%%%%%%%%%%
\section{Homology of sub-level sets}
%%%%%%%%%%%%%%%%%%%%%%%%%%%%%%%%%%%%%%%%%%%%%%%%%%%%%%%%%%%%%%%%%%

\label{sec::HomologyOnCW}

\begin{figure}[tb]
\subfigure[]{\includegraphics[width = .38\textwidth]{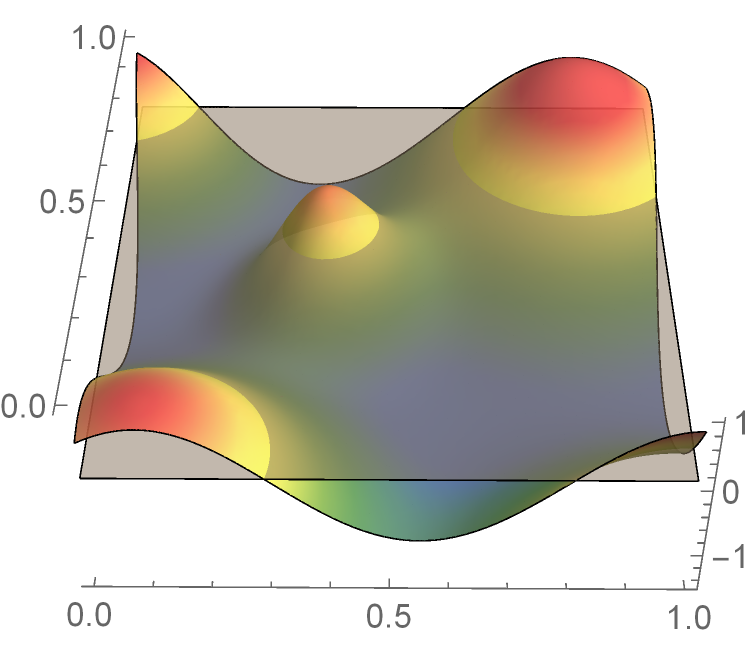}}
\subfigure[]{\includegraphics[width = .295\textwidth]{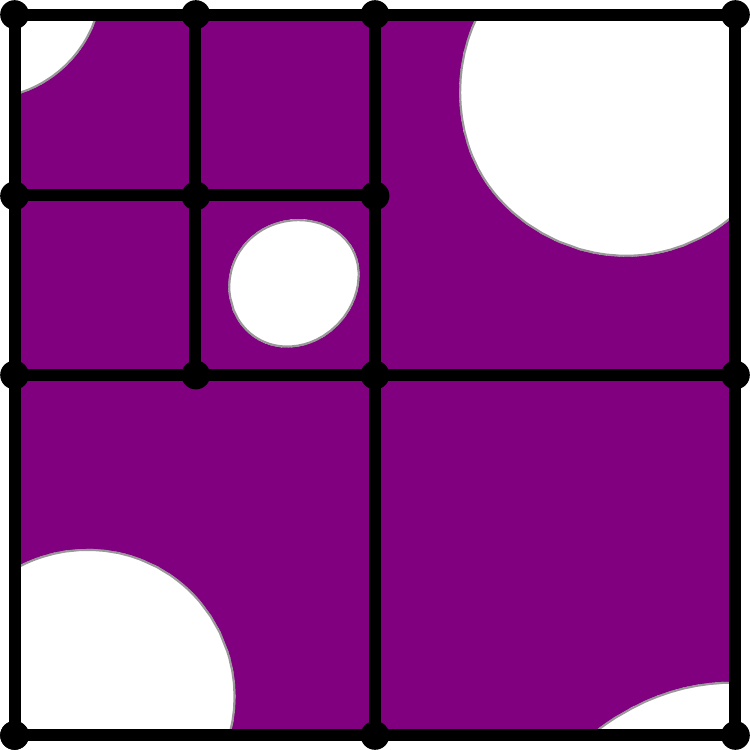}}
\subfigure[]{\includegraphics[width = .3\textwidth]{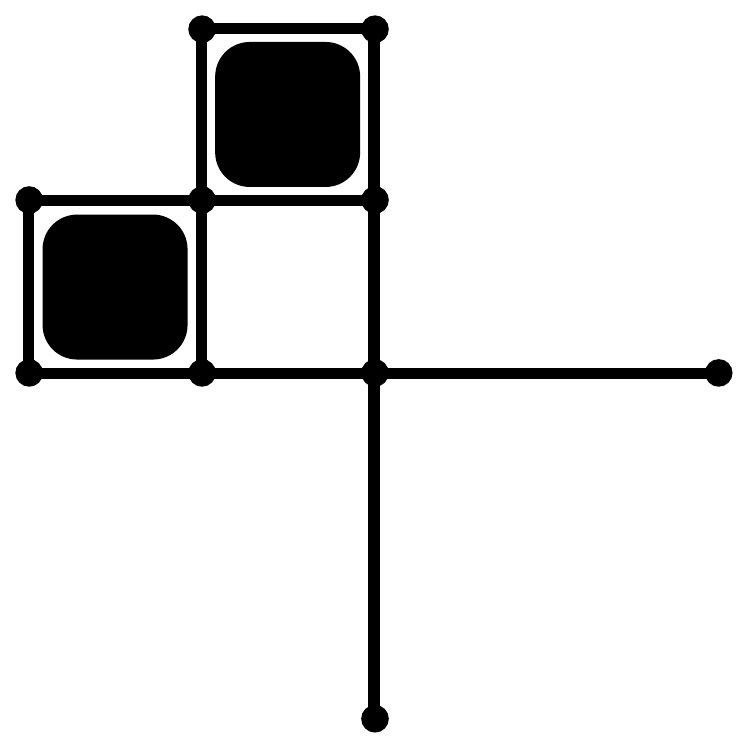}}
\caption{
(a) A simple function $f : [0,1]^2 \to \R$. The sub-level set $X_{0} = f^{-1}((-\infty,0])$ corresponds to the part of  $[0,1]^2$ at which  the value of $f$ is below the semitransparent plane. 
(b) The set $X_{0}$ is shown in purple. The black vertices, edges and squares bounded by the edges define a CW structure $\cE$ on $[0,1]^2$. 
(c) The set $\cX_{0}$ corresponding to the decomposition $\cE$.}
\label{fig::function}
\end{figure}

In this section we present a method for computing the homology of sub-level sets. 
Let $X$ be a CW-complex and   $f : X \to \R$ a continuous function; we are interested in the homology of the sub-level set  $X_t := f^{-1}(-\infty, t]$. Suppose that $\cE$ is a CW structure on $X$, then we define the \emph{cellular approximation} to be 
\begin{equation}
\cX_t : = \bigcup\setof{ e\in \cE : f(\cl(e)) \leq t}.
\end{equation}
Note that $\cX_t$ is a closed subset of $X_t$, and since $\cX_t$ is a finite union of cells, then it is a sub-complex of $X$. 
For the function $f$ shown in Figure~\ref{fig::function} the set $X_{0}$ is homologous to $\cX_{0}$.   
According to Theorem~\ref{prop::IntroCorrect_homology},  the homology groups of $X_t$ and $\cX_t$ are isomorphic if the CW-structure $\cE$ is compatible with $X_t$. We close this section by proving Theorem~\ref{prop::IntroCorrect_homology}.

\begin{proof}

If   $H_*(X_t,\cX_t) \cong 0$  (using singular homology), then the fact that $i_*$ is an isomorphism  follows from  the long exact sequence of the pair. We prove that  $H_*(X_t,\cX_t) \cong 0$ by induction on the dimension of $X$. 

For a zero dimensional complex $X$, the set $ X_t$ is equal to $\cX_t$ and so trivially $H_*(X_t,\cX_t) \cong 0$. Now we show that   $H_*(X_t,\cX_t) \cong 0$  for an  $n$ dimensional complex $X$ under the assumption that the statement is true for any complex with dimension less than $n$. 
Let us first suppose that $X$ contains a single $n$ dimensional cell $e$. 
Then $X = X^{n-1}\cup e$. Define $\cX^{n-1}_t$ to be the $n-1$ skeleton of $\cX_t$.
While $X_t$ is not a CW complex in general, we define an analog of its $n-1$ skeleton by $X_t^{n-1} := X_t \cap X^{n-1}$.  
By the inductive hypothesis, the homology groups $H_*(X_t^{n-1},\cX^{n-1}_t) $ are trivial. 

The CW structure $\cE$ is compatible with $X_t$,  so the cell $e$  satisfies one of the conditions in Definition~\ref{def::CompatibleCW}.
If $f(\cl(e)) > t$, then $X_t  = X^{n-1}_t $ and $\cX_t = \cX^{n-1}_t $. 
Hence we obtain the set equality $H_*(X_t,\cX_t)  =H_*(X_t^{n-1},\cX^{n-1}_t) $ and so by the inductive assumption the relative homology groups stay trivial.

Now suppose that $f(\cl(e)) \leq t$. 
In this case both $\cl(e) \subseteq X_t$ and $\cl(e) \subseteq \cX_t$. 
In general, we have the equality $(X_t \backslash e , \cX_t \backslash e) = ( X_t^{n-1}, \cX_t^{n-1})$, so it follows from the inductive assumption that $ H_*(X_t \backslash e , \cX_t \backslash e) \cong 0$.
We would like to apply the excision theorem to obtain $H_* (X_t , \cX_t ) \cong H_* (X_t \backslash e , \cX_t \backslash e)  \cong 0$.  
However, the assumptions of the singular excision theorem are not actually satisfied because $\cl(e)$ is not in the interior of $\cX_t $. This problem can be overcome by expanding $X_t$ and $\cX_t$  inside $e$ using a small $\varepsilon$ neighborhood of $\bd(e)$. Note that for $\varepsilon$ small enough the homology groups of $X_t$ and $\cX_t$ do not change.

If $e$ satisfies  Condition (3) in  Definition~\ref{def::CompatibleCW}, then there exists a deformation retraction  $h(\bx,s)$ of the set $X_t \cap \cl(e)$ onto $X_t \cap \bd(e)$. 
Note that $h$ is a continuous map on the closed set $X_t \cap \cl(e) \times [0,1]$, and the identity map on the closed set $X_t \backslash e  \times [0,1]$ is also continuous. 
Since these two maps agree on the intersection of their domains, we can extend $h$ to a deformation retraction of $X_t $ onto  $X^{n-1}_t$. 
Existence of $h$ also implies that there exists some $t' > t$ such that $t' \in f(\cl(e))$, hence $e \cap \cX_t = \emptyset$ and $\cX_t = \cX_t^{n-1}$. By induction   $H_*(X_t,\cX_t) \cong 0$ again.

Finally suppose that  $X$ is an arbitrary  $n$ dimensional complex. Let  $e_1, \ldots, e_m$  be its $n$ dimensional cells. We repeat the above process $m$ times. At each step we  add a cell $e_k$ to $X^n \cup e_1, \cup \ldots \cup e_{k-1}.$ 

\end{proof}

%%%%%%%%%%%%%%%%%%%%%%%%%%%%%%%%%%%%%%%%%%%%%%%%%%%%%%%%%%%%%%%%%%
\section{Persistent homology}
%%%%%%%%%%%%%%%%%%%%%%%%%%%%%%%%%%%%%%%%%%%%%%%%%%%%%%%%%%%%%%%%%%
\label{sec::PD}

In order to exactly compute the persistence diagrams $\pd(f)$ of a continuous bounded function, one must identify all thresholds $t_i$ for which the homology of $X_{t_{i}- \epsilon}$ and $X_{t_{i}+ \epsilon}$ differ for all $0 < \epsilon <<1$. If $f$ is a Morse function, then the thresholds $t_i$ are the critical values of $f$, which one may identify after locating every point where the Jacobian $Df$ vanishes. 
Such a root finding problem is frequently studied in nonlinear optimization \cite{hansen2003global}. 
In applications though, it may not always be feasible to rigorously find every root of $Df$.  
In such a situation, it would be desirable to still be able to compute the persistence diagrams of $f$, even at the cost of losing some accuracy.

In this section we present a  framework  for computing an approximation  of the persistence diagram $\pd(f)$.
For a given $\varepsilon > 0$ we  construct a cellular approximation of the filtration corresponding to $f$. The persistence diagram $\pd'$ of this approximation is $\varepsilon$ close to $\pd(f)$, that is $d_B(\pd(f),\pd') < \varepsilon$. We start by choosing an approximate filtration (cf. Definition \ref{def::ApproximateFiltration}). 
The following lemma shows that the bottleneck distance between the persistent diagram for the $\varepsilon$  approximate filtration of $f$ and the persistence diagram of $f$ is at most $\varepsilon$.

\begin{lemma}
\label{prop::closeApproximation}
Let $X$ be a CW complex and $f : X \to \R$ a continuous bounded function. If $\pd'$ is a persistence diagram of any  $\varepsilon$ approximate filtration of $f$ then  $d_B(\pd(f),\pd')  < \varepsilon$.
\end{lemma}

\begin{proof}
To prove the lemma we have to show that for every $n$-th persistence diagram $\pd_n(f)$ there exists a  bijection $\gamma$ from $\pd_n(f)$ to $\pd'_n$ such that $\| p - \gamma(p)\|_\infty < \varepsilon$ for every $p \in \pd(f)$.

If $p=(b,d)$ and $d - b < \varepsilon$, then we define $\gamma(p) = (b,b)$. Clearly  $\| p - \gamma(p)\|_\infty < \varepsilon$ in this case.
The collection of sets $\{ X_{s_i} \}_{i=-1}^{m+1}$ is an $\epsilon$ filtration. Therefore, if $d - b \geq \varepsilon$, then there exist indices $i$ and $j$  such that  $s_{i-1} < b \leq s_i \leq s_{j-1} < d \leq s_j$.
So the topological feature corresponding to $p = (b,d)$ is not present in the sub-level set $X_{s_{i - 1}}$ but it can be found in $X_{s_i}$. Hence for the $\varepsilon$ approximate filtration there exists  a unique topological  feature $q$, corresponding to $p$,  born at $s_i$. By the same reasoning the feature $q$ dies at $s_{j}$ and we define $\gamma(p) = (s_{i},s_{j})$. Again the distance $\| p - \gamma(p)\|_\infty < \varepsilon$, and the number of points in $\pd'$ is smaller or equal to the number of points in $\pd(f)$. So $\gamma$ is a bijection.  
\end{proof}

It is appealing to construct a  cellular  approximate filtration by  replacing  the sets $X_{s_i}$ with their cellular surrogates $\cX_{s_i}$ introduced in the previous section. 
However the collection $\setof{\cX_{s_i}}_{i = -1}^{m+1}$ might not be a filtration, as  shown in Figure~\ref{fig::NotAFiltration}. Namely, the set $\cX_0$ is not a subset of $\cX_1$ because of  different choices of the compatible CW structures  for the sub-level sets.
To avoid this problem we construct a common refinement of the CW structures. Our  refinement  process  relies on the fact that  the CW structures $\cE_{s_i}$ are commensurable (cf. Definition \ref{def::commensurable}).

\begin{definition}
Let $\setof{\cE_{s_i}}_{i = -1}^{m+1}$ be a  commensurable collection of CW structures on $X$. Then the  set 
\[
\cE := \setof{ e \in \cE' :  \text{ either } e\cap e' = \emptyset \text{ or } e\subseteq e' \text{ for every } e' \in \cE' }
\]
is called the  \emph{common refinement} of the CW structures  $\setof{\cE_{s_i}}_{i = -1}^{m+1}$.
\
\label{def::refinement}
\end{definition}

The CW structures $\cE_{0}$ and $\cE_{1}$ shown in Figures~\ref{fig::NotAFiltration} (b-c) are commensurable and their common refinement is $\cE_{0}$. In this paper we will use the common refinement of all the structures $\setof{\cE_{s_i}}_{i = -1}^{m+1}$ to compute the persistent homology. 
If zig-zag persistence \cite{Carlsson:2009:ZPH:1542362.1542408} is used, then it suffices to consider  the common refinements of consecutive CW structures $\cE_{s_i}$ and $\cE_{s_{i+1}}$.
We stress that the common refinement does not need to be compatible with any  of the sub-level sets $X_{s_i}$. Hence we need to  develop a more sophisticated approach for building  a cellular representation of the approximate filtration. 
Before  introducing a cellular representation of the filtration $\setof{X_{s_i}}_{i =-1}^{m+1}$  let us prove that the common refinement $\cE$ is a  regular  CW structure on $X$.

\begin{lemma}
Let $\cE$ be  a common refinement of the commensurable CW structures  $\setof{\cE_{s_i}}_{i =-1}^{m+1}$ on $X$. Then $\cE$ is a CW structure on $X$.
\label{lem:ninimalCW}
\end{lemma}

\begin{proof}
First we show that every point $x \in X$ is contained in a unique cell $e \in \cE$. Note that the set $C(x) = \setof{ e \in \cE'  : x \in e}$ is not empty. 
If $e,e'\in C(x)$, then by  Definition~\ref{def::commensurable}  either $e\subseteq e'$ or $e'\subseteq e$. This implies that the cell  $e_x := \bigcap\setof{ e \in C(x)} $ is an element of $\cE'$. Moreover for any $e\in C(x)$ the cell $e_x \subseteq e$ and $ e \cap e_x = \emptyset$ if $e \not\in C(x)$.  Therefore  $e_x \in \cE$. Suppose by contradiction that $e_x$ is not a unique cell in $\cE$ containing $x$. Then there exist $e,e' \in \cE$ such that  $x\in e \cap e'$. It follows from Definition~\ref{def::refinement} that $e\subseteq e'$ and $e'\subseteq e$. So $e = e'$.

We showed that $\cE$ is a disjoint cellular decomposition of $X$. Naturally, every cell $e \in \cE \subseteq \cE'$ is homeomorphic to some euclidean space $\R^{\dim (e)}$.  Now we need to define the characteristic maps for the cells. Every cell in $e \in \cE$ belongs to at least one   $\cE_{s_i}$. We identify the characteristic map for $e \in \cE$ with the characteristic map given by one (arbitrary but fixed) CW structure $\cE_{s_i}$.  Since every cell in $\cE '$ is covered by a collection of smaller cells in $\cE$,  the boundary of a cell $ e \in \cE$ is contained in the union of cells in $\cE$ of dimension less than $\dim (e)$. Hence the set of cells  $\cE$  equipped with these maps is a CW structure on $X$.
\end{proof}

\begin{figure}[tb]
\subfigure[]{\includegraphics[width = .24\textwidth]{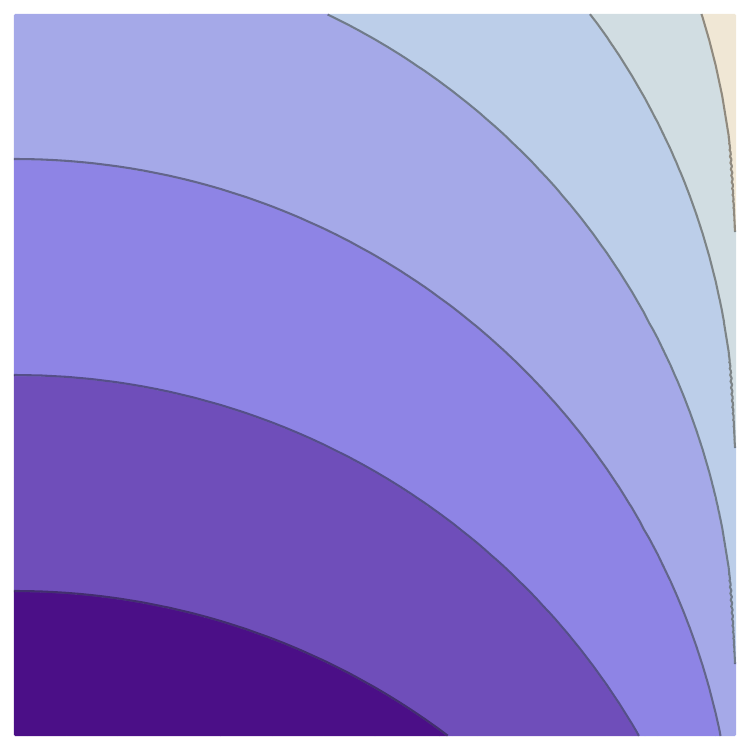}}
\subfigure[]{\includegraphics[width = .24\textwidth]{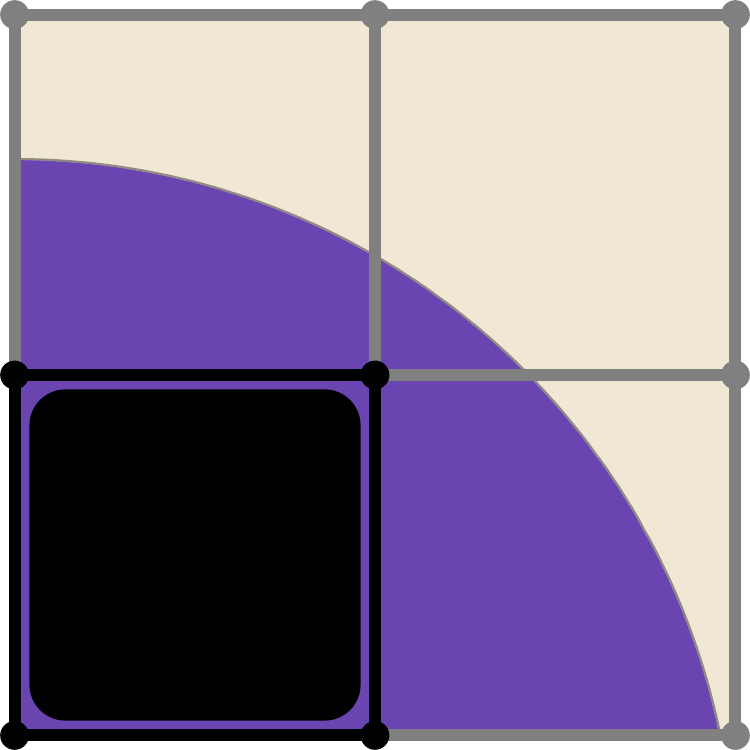}}
\subfigure[]{\includegraphics[width = .24\textwidth]{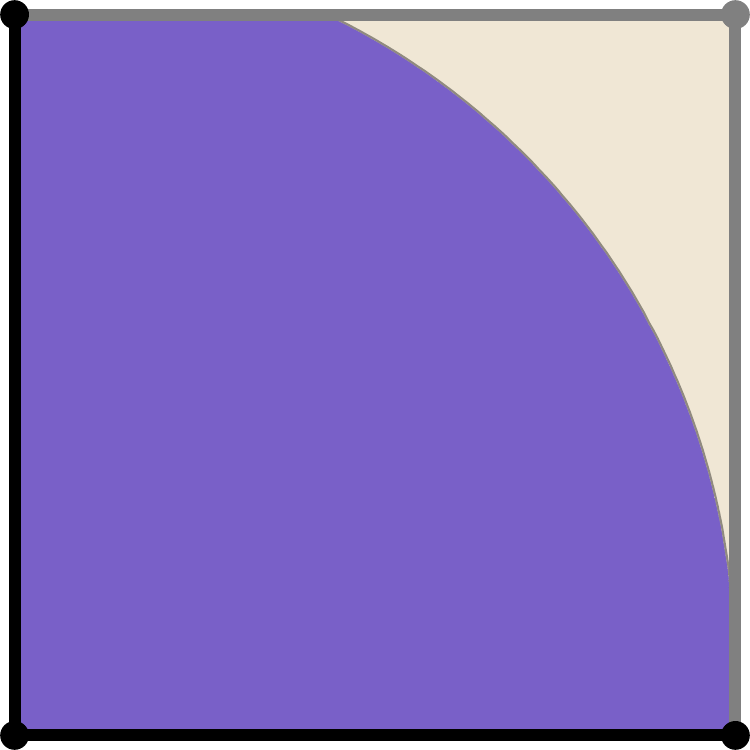}}
\subfigure[]{\includegraphics[width = .24\textwidth]{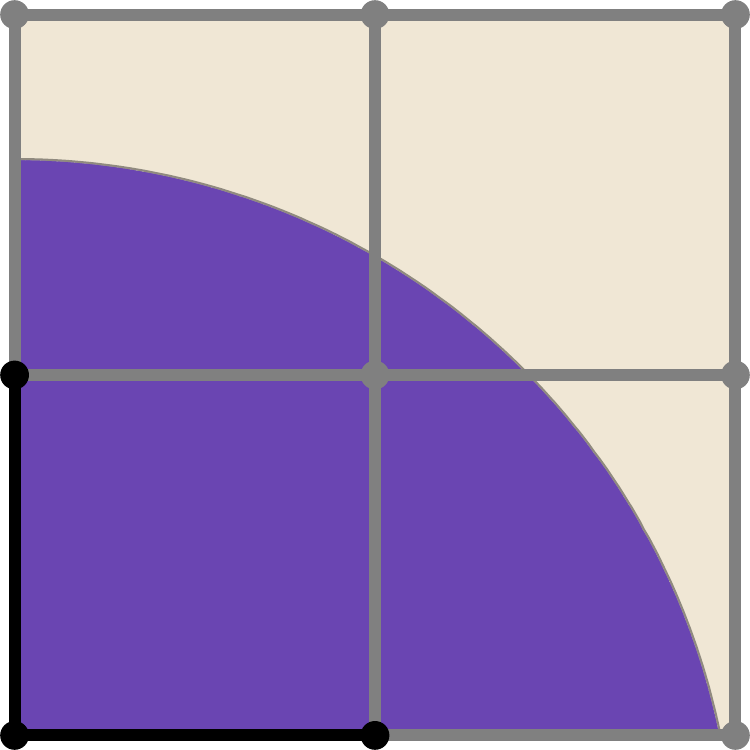}}
\caption{(a) A contour plot of  a  function $f : [0,1]^2 \to \R$.  (b) The sub-level set  $X_{0}$ and its cellular approximation  $\cX_{0}$  obtained by the compatible CW structure with four squares ($2$ dimensional cells.) (c) The sub-level set   $X_{1}$ and its cellular approximation  $\cX_{1}$  obtained by the compatible CW structure  with a single $2$ dimensional cell. (d)  The sub-level set $X_0$ with the refined approximation $\bar{\cX}_{0}$, which is equal to $ \cX_{0} \cap \cX_{1}$.  }
\label{fig::NotAFiltration}
\end{figure}

Since the common refinement $\cE$ does not have to be compatible with any of the sub-level sets, we cannot  use it to generate a cellular approximation of the filtration $\setof{X_{s_i}}_{i=-1}^{m+1}$.  
Instead we approximate the sub-level sets $X_{s_i}$ by 
\begin{equation}
\label{eqn::filtarion}
\bar{\cX}_{s_i} = \bigcap_{j \geq i } \cX_{s_j}. 
\end{equation}
\begin{definition}
Let $\setof{X_{s_i}}_{i = -1}^{m+1}$ be an $\varepsilon$  approximate filtration of a continuous function $f$ and $\setof{\cE_{s_i}}_{i = -1}^{m+1}$ a collection of commensurable CW structures, where  $\cE_{s_i}$ is compatible with $X_{s_i}$ for every $s_i$.  Then the collection of sets $\setof{\bar{\cX}_{s_i}}_{i = -1}^{m+1}$ is called an $\varepsilon$ \emph{approximate cellular filtration} of $f$ corresponding to $\setof{\cE_{s_i}}_{i = -1}^{m+1}$.
\end{definition}

\begin{remark}
\label{rem::Inclusions}
Every set $\bar{\cX}_{s_i}$ is a sub-complex of $X$, where the CW structure of $X$ is given by the common refinement $\cE$. 
Moreover   $\setof{\bar{\cX}_{s_i}}_{i = -1}^{m+1}$ is a filtration, that is
\[
\emptyset = \bar{\cX}_{s_{-1}} \subseteq \bar{\cX}_{s_0} \subseteq \dots \subseteq \bar{\cX}_{s_{m+1}} = X.
\]
\end{remark}

Our goal is to show that the persistence diagrams of the filtrations $\setof{{X}_{s_i}}_{i = -1}^{m+1}$  and $\setof{\bar{\cX}_{s_i}}_{i = -1}^{m+1}$ are equal. This is true  if  the diagram shown below in Figure \ref{fig::CommutativeDiagram}, with the maps  induced by inclusion,  commutes and the maps $j^i_*$ are isomorphisms.  It follows from Remark~\ref{rem::Inclusions} and the functorality of homology that the diagram commutes. To prove that $j^i_*$ are isomorphisms we will have to strengthen the definition of a compatible CW structure. 

\begin{figure}[H]
\[
\begin{CD}
\dots 	@>>> H_* ( X_{s_{i-1}} ) 					@>>>  	H_* ( X_{s_i} ) 				@>>>  	H_* ( X_{s_{i+1}})  				@>>> \dots \\
@.					@AA j^{i-1}_* A													@AA j^{i}_* A    									@AA j^{i+1}_*A @. \\
\dots 	@>>> H_* ( \bar{\cX}_{s_{i-1}} ) 	@>>> 		H_* ( \bar{\cX}_{s_i} ) 	@>>>  H_* ( \bar{\cX}_{s_{i+1}})  @>>> \dots 
\end{CD}
\]
\caption{The persistence diagrams of the filtrations $\setof{X_{s_i}}_{i = -1}^{m+1}$ and  $\setof{\bar{\cX}_{s_i}}_{i = -1}^{m+1}$  are equal if the above diagram commutes and the maps $j_*^i$, induced by inclusion, are isomorphisms. }
\label{fig::CommutativeDiagram}
\end{figure}

First, we illustrate the idea on a simple example shown in Figure~\ref{fig::NotAFiltration}.
The cells $e \subset \cX_{0}$ that are not in $\bar{\cX}_{0}$ are inside of the single 2 dimensional cell $e'  : = (0,1)^2 \in \cE_{1}$ and additionally $\bar{\cX}_{0} = \cX_{0} \cap \bd(e')$. 
Furthermore $e' \cap \cX_{1} = \emptyset$ because the image $f(\cl(e'))$ contains both values above and below $1$. The fact that the CW structure $\cE_1$ is compatible with $X_1$ implies that there is a deformation retraction $h (\bx, s)$ from $X_1 \cap \cl(e')$ onto $X_1 \cap \bd(e')$. If $f \circ h(\bx,s)$ is non-increasing in $s$, then it can be restricted to a deformation retraction from $X_0 \cap \cl(e')$ onto $X_0 \cap \bd(e')$. By the same argument as in the proof of Theorem~\ref{prop::IntroCorrect_homology}  the inclusion map from $\bar{\cX}_{0} = \cX_0 \cap \bd(e')$ to $X_0$ induces an isomorphism on the homology level.  
We now prove Theorem \ref{prop::ApproximateFiltration},  which treats the general case.

\begin{proof}
The persistence diagrams of the filtrations $\setof{{X}_{s_i}}_{i = -1}^{m+1}$  and $\setof{\bar{\cX}_{s_i}}_{i = -1}^{m+1}$ are equal if the diagram in Figure~\ref{fig::CommutativeDiagram} commutes and the maps $j_*$, induced by inclusion, are isomorphisms. We already showed that the diagram computes. Now we have to prove that every $j^i_* \colon H_*(\bar{\cX}_{s_i}) \to H_*(X_{s_i})$ is an isomorphism. According to  Theorem~\ref{prop::IntroCorrect_homology}, the inclusion of each $\cX_{s_i}$ into $X_{s_i}$ induces an isomorphism.  Hence it suffices to show that the inclusion from $\bar{\cX}_{s_i}$ into $\cX_{s_i}$ also induces an isomorphism in homology.

The set $\bar{\cX}_{s_i}$ is a sub-complex of $\cX_{s_i}$. So
\[\bar{\cX}_{s_i} = \cX_{s_i} \setminus \setof{e_1 \cup e_2 \cup \ldots \cup e_N},
\]
for some cells $e_a \in \cE$. We denote $E =  \setof{e_1, e_2, \ldots, e_N}$. First we divide the  cells in $E$ into disjoint groups. Then we show that by removing all of the cells in the first group we obtain a CW sub-complex of $\cX_{s_i}$ and the map induced by the  inclusion of this complex into  $\cX_{s_i}$ is an isomorphism.  We prove that the maps  $\bar{j}^i_* \colon H_*(\bar{\cX}_{s_i}) \to H_*(\cX_{s_i})$  are isomorphisms by inductively removing the cells in all the other groups.

Before grouping the cells into disjoint classes let us show that for every cell $e_a \in E$ there exists $e'_a \in \cE_{s_i}$ and $e''_a \in \cE_{s_j}$, where $j > i$, such that: 
\begin{enumerate}
\item $e_a \subseteq e'_a \subsetneq e''_a$,
\item $e''_a \cap \cX_{s_j} = \emptyset$.
\end{enumerate} 
By the construction of $\cE$, each cell $e_a$ is a subset of a unique cell $e'_a \in \cE_{s_i}$ and $e'_a \subseteq \cX_{s_i}$. 
The fact that $e_a \cap \bar{\cX}_{s_i} = \emptyset$ implies  $e_a \cap {\cX}_{s_j} = \emptyset$ for some $j > i.$ It follows from the construction of the cellular approximation $\cX_{s_j}$ that  $e'_a \not\in \cE_{s_j}$. Otherwise we would have $e_a \subseteq e'_a \subseteq \cX_{s_j}$. Therefore there exists a cell $e''_a \in \cE_{s_j}$ which satisfies  the  properties (1) and (2).

Let $\setof{ e''_1, \ldots, e''_N}$ be the cells with the properties (1) and (2)   corresponding to the cells in $E$ and let $E '' = \setof{ e''_1, \ldots, e''_M} \subseteq \setof{ e''_1, \ldots, e''_N}$ be a set of the maximal cells. That is, a cell $e''_i$ is removed from the set $\setof{ e''_1, \ldots, e''_N}$  if there is another cell $e''_j$ such that $e''_i \subseteq e''_j$. We order the cells in $E''$ by dimension so that $\dim(e''_i) \geq \dim( e''_j)$ for $i \leq j$. We say that the cells  $e_i , e_j \in E$ are in the same group if there exists some $l$ such that both $e_i$ and $e_j$ are subsets of $e''_l$.

We will remove the cells   $e \in E$ from $\cX_{s_i}$ in $M$ steps. 
During the $l$-th step all of the cells $e \in E$   such that $ e \subset e''_l$ are removed. Note that   $\bar{\cX}_{s_i} = \cX_{s_i} \setminus  \cup_{a = 1}^M e_a''$.  Due to the ordering of the cells in $E''$ we get a CW-complex at each step. 
To show this let  $e \in E$  be contained in some  $e''_l \in E''$. We show that $e$ is not in the boundary of any cell in $\cX_{s_i} \backslash \cup_{a=1}^{\ell} e_a''$. The cell $e''_l \in \cE_{s_j}$ for some $j > i$ and $e''_l \not\in \cX_{s_j}$.  It follows from the construction of $\cX_{s_j}$ that every cell $c \in \cE_{s_j}$ such that $e''_l \cap \cl(c) \neq \emptyset$ is not in $\cX_{s_j}$. So every cell in  $\cX_{s_i}$ 
 whose closure intersects the cell $e$ is contained in some cell  $e''_k \in E''$ such that  $\dim(e''_k) > \dim(e''_l)$  and was removed before $e$. Therefore at each step we get a closed subset of the cells from  $\cX_{s_i}$ which is a CW-complex. 
 Moreover, at each step we can express the complex as a finite union of cells from $\cE_{s_i}$, so it also has a CW structure induced by $\cE_{s_i}$.

First we suppose that $l = 1$ and prove that the map  $\ell_* \colon  H_*(\cX_{s_i} \setminus  e_1'') \to H_*(\cX_{s_i})$, induced by  inclusion, is an isomorphism. To do so we apply  the five lemma to the following diagram.

\begin{figure}[H]
\[
\begin{CD}
@>>> H_* (  \cX_{s_i} \cap \bd(e_1'')) 					@>>>    H_* ( \cX_{s_i} \setminus e_1''   ) \oplus H_* ( \cX_{s_i} \cap \cl(e_1'')  ) 				@>>>  H_* ( \cX_{s_i}  )  @>>> \\
@. 						@AA  A						@AA  A    								@AA \ell_* A @. \\
@>>> H_* ( \cX_{s_i} \cap \bd(e_1'') ) 		@>>>   H_* ( \cX_{s_i} \setminus e_1''  ) \oplus H_* ( \cX_{s_i} \cap \bd(e_1'')  ) 	@>>>  H_* ( \cX_{s_i} \setminus e_1'' )  @>>>\
\end{CD}
\]
\end{figure}

%%%%%%%%	Here is another way to generate the commutative diagram. 

%\begin{center}
%\begin{tikzcd}
%%\arrow{r}  & 
%H_*(  \cX_{s_i} \cap \bd(e_1''))  \arrow{r} & H_* ( \cX_{s_i} \setminus e_1''   ) \oplus H_* ( \cX_{s_i} \cap \cl(e_1'')  )  \arrow{r} & H_*( \cX_{s_i}  )  \arrow{r} & \ldots  \\
%%\arrow{r}  &   
%H_*(\cX_{s_i} \cap \bd(e_1'')) \arrow{u} \arrow{r} & H_* ( \cX_{s_i} \setminus e_1''   ) \oplus H_* ( \cX_{s_i} \cap \bd(e_1'')  ) \arrow{u}	\arrow{r} & H_* ( \cX_{s_i} \backslash e_1'') \arrow{u}{\ell_*}   \arrow{r}  & \ldots 
%\end{tikzcd}
%\end{center}

Horizontal maps  are given by the cellular Mayer-Vietoris sequences for the CW complexes given by the middle terms. 
The vertical maps are induced by inclusion.  
 The first two vertical maps are induced by identity maps and are thereby isomorphisms. 
To satisfy the assumptions of the five lemma we have to show that the map  $k_* \colon H_* ( \cX_{s_i} \cap \bd(e_1'')) \to  H_* ( \cX_{s_i} \cap \cl(e_1''))$  is also an isomorphism. 

We recall that the cell $e_1'' \in \cE_{s_j}$ and $e_1'' \cap \cX_{s_j} = \emptyset$ for some $j > i$. Therefore the image $f(e_1)$ contains points both above and below $s_j$ and there has to exist a deformation retraction $h(\bx,s) \colon  X_{s_j} \cap \cl( e_1'') \to X_{s_j} \cap \bd(e_1'')$. 
We have assumed that $f \circ h(\bx,s)$ is non-increasing in $s$. 
So $h$ restricts to a deformation retraction from $X_{s_i} \cap \cl( e_1'')$ onto $X_{s_i} \cap \bd(e_1'')$.  
Hence the top horizontal map in the diagram below is an isomorphism. 
\begin{figure}[H]
\[
\begin{CD}
H_* (X_{s_i} \cap \bd(e_1'') ) 					@> \cong >>  H_* ( X_{s_i} \cap \cl(e_1'') ) 			\\
@AA \cong A											@AA \cong A    						\\
H_* ( \cX_{s_i} \cap \bd(e_1'') ) 	@>k_*>> H_* ( \cX_{s_i} \cap \cl(e_1'') ) 	
\end{CD}
\]
\end{figure}
There exist regular CW decompositions of $\bd(e_1'')$ and $\cl(e_1'')$ comprised of cells from $ \cE_{s_i}$. It follows from Theorem  \ref{prop::IntroCorrect_homology}  that the vertical arrows in the diagram above are isomorphisms. By the commutativity of the diagram, it follows that the map $k_*$ is an isomorphism.
So by the five lemma, it follows that the map $\ell_* : H_* ( \cX_{s_i} \setminus e_1'' ) \to  H_* ( \cX_{s_i} \setminus e_1'' )$ is an isomorphism. 

Using mathematical induction we suppose that the inclusion map from $\cX_{s_i} \setminus  \cup_{a = 1}^l e_a''  $ to $\cX_{s_i}$ induces an isomorphism on the homology level. To show that this is also true for $l+1$ we repeat the above process replacing the set $\cX_{s_i}$ by $\cX_{s_i} \setminus  \cup_{a = 1}^l e_a'' $ and $e_1''$ by $e_{l+1}''$. 
\end{proof}

%%%%%%%%%%%%%%%%%%%%%%%%%%%%%%%%%%%%%%%%%%%%%%%%%%%%%%%%%%%%%%%%%%
\section{ Cellular complexes on grid}
%%%%%%%%%%%%%%%%%%%%%%%%%%%%%%%%%%%%%%%%%%%%%%%%%%%%%%%%%%%%%%%%%%
\label{sec:grid}

In order to build a cellular approximation  $\cX_t$ of $X_t$ we need to construct commensurable CW structures on the space $X$. We do this  in an adaptive manner, starting with a coarse CW structure $\cE$ on $X$.   To obtain a cellular approximation which is compatible with a given sub-level set,  the cells that do not satisfy any of the conditions in Definition~\ref{def::CompatibleCW} are repeatedly refined. 
If the critical values of $f$ are bounded away from $t$, then this process terminates in a finite number of steps.

To formalize this process we introduce the notion of a self similar grid. 
Using this type of grid also facilitates a memory efficient method for storing the complex  and its boundary operator, see Section~\ref{sec:DataAndAlgorithms}.  
At the end of this section we show how to compute the boundary operator for the refined  CW structure on $X$ in terms of the incidence numbers for the basis $\Lambda$ given by  $\cE$.

\begin{figure}[]
\subfigure{\includegraphics[height = .315\textwidth]{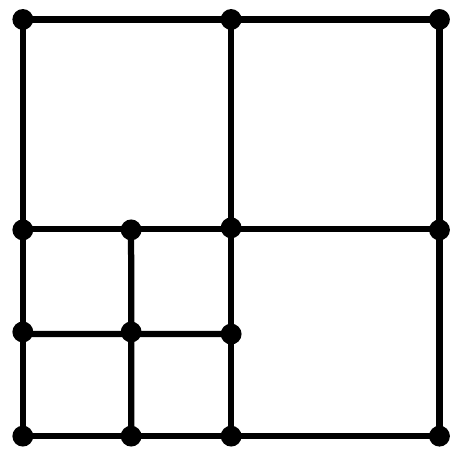}}
\subfigure{\includegraphics[height = .32\textwidth]{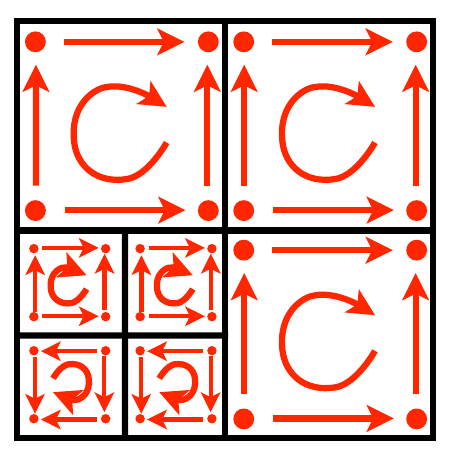}}
\subfigure{\includegraphics[height = .32\textwidth]{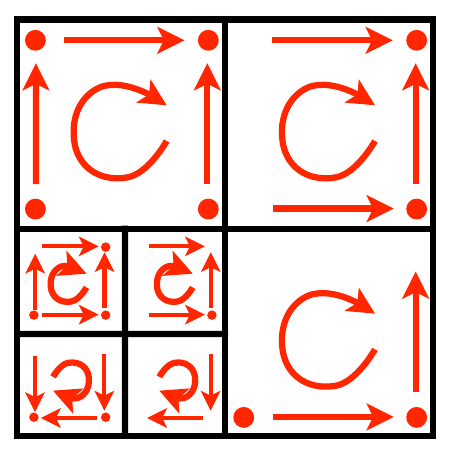}}
\caption{(a) A  dyadic   grid $\cG$ on a unit square. (b)  CW structures  on the  dyadic cubes.  (c) CW structure $\cE_\cG$ induced by $\cG$. }
\label{fig:DyadicGrid}
\end{figure}

Before defining a  self similar grid in full generality we consider the case of a unit cube $[0,1]^N$. 
We choose $\cE$ to be the standard CW structure on the cube consisting of the vertices, edges, higher dimensional faces and a single N-dimensional cell. 
A self similar grid for the unit cube is built out of dyadic cubes defined as follows.

\begin{definition}
Let $n$ be an arbitrary natural number. For every sequence  $\setof{m_i}_{i=1}^N$, where $0 \leq m_i < 2^n$,  we define an  \emph{$N$ dimensional dyadic cube} as  a closed set 
\[
C:= \prod_{i=1}^N{ \left[\frac{m_i}{2^n} , \frac{m_i+1}{2^n} \right]}.
\]
\end{definition}

Every dyadic cube is homeomorphic to $[0,1]^N$. 
Compositions of this  homeomorphism with the characteristic maps for the cells $e \in \cE$ define a CW structure on the dyadic cubes.  
A dyadic  grid is a particular self similar grid consisting of dyadic cubes.

\begin{definition}
Let $\cG = \{C_i\}_{i=0}^M$ be a finite collection of $N$ dimensional dyadic cubes.  Then $\mathcal G$ is a \emph{dyadic  grid} if the following conditions are satisfied:
\begin{enumerate}
	\item $[0,1]^N = \bigcup_{i=1}^n C_i $,
	\item $int (C_i) \cap int(C_j) = \emptyset $ if $i \neq j$.
\end{enumerate}
\end{definition}

Figure~\ref{fig:DyadicGrid}(a) shows a simple example of dyadic grid on a unit square. 
The CW structures of the dyadic cubes are depicted  in Figure~\ref{fig:DyadicGrid}(b). 
Note that  orientation of a cell depends on the precise choice of the homeomorphism $h_i \colon [0,1]^N \to C_i$. 
In our example we reversed orientation of the cells in two cubes at the bottom left corner. 

We generalize the idea of  dyadic grids to CW complexes which are homeomorphic to a closed ball. 
In the rest of this paper we assume that there is a regular  CW structure $\cE$   on  $X$  with a single $N$ dimensional cell $e$ such that $X = \cl(e)$, thus implying that $X$ is homeomorphic to an $N$ dimensional closed ball.
The building blocks of a self similar grid are  homeomorphic to $X$ and their CW structure is again induced by these homeomorphisms.  
They also have to cover $X$ in a similar  manner as the dyadic cubes cover the unit cube. 
A self similar grid on $X$ is   a special collection of the building blocks.

\begin{definition}
Let $X$ be an $N$ dimensional regular CW complex with a single $N$ dimensional cell $e$ such that $X = \cl(e)$. The set $\cB$ is called a \emph{collection of building blocks} of $X$  if it satisfies the following conditions:
\begin{enumerate}
\item An element $C \in \cB$ is a CW complex homeomorphic to $X$, and $C \subseteq X$.
%
%\item  Given building blocks $C_1, C_2 \in \cB$ select cells $e_i \subset C_i $ for $i=1,2$.  If $e_1 \cap e_2 \neq \emptyset$, then either $e_1\subseteq e_2$ or $e_2\subseteq e_1$.
\item   Let   $C_1, C_2 \in \cB$ and $e_i \subset C_i $ for  $i=1,2$. If $e_1 \cap e_2 \neq \emptyset$, then either $e_1\subseteq e_2$ or $e_2\subseteq e_1$.
\item $X \in \cB$.
\end{enumerate}
\label{def:BuildingBlocks}
\end{definition}

\begin{definition}
Let $\cB$ be a  collection of building blocks of $X$. Then $\cG \subseteq \cB$ is a \emph{self similar}  grid on $X$  if the following conditions are satisfied:
\begin{enumerate}
\item $X = \bigcup_{C \in \cG} C $,
\item For all $C_1,C_2 \in \cG$ the intersection $int (C_1) \cap int(C_2) = \emptyset $ if $C_1 \neq C_2$.
\end{enumerate}
\label{def:SelfSimilarGrid}
\end{definition}

Using the same arguments as in the proof of Lemma~\ref{lem:ninimalCW}, one can  show that by removing the redundant cells, a self similar grid $\cG$ on $X$ generates a CW structure $\cE_\cG$. 
For our simple example of a dyadic grid, this structure is shown in Figure~\ref{fig:DyadicGrid}(c).    
If we fix a collection of building blocks $\cB$, then a collection of CW structures  $\setof{\cE_{\cG_i}}$ generated  by the self similar grids $\setof{\cG_i}$ is always commensurable. The following two lemmas  formalize the above statements.

\begin{lemma}
Let $\cG = \setof{C_i}_{i=0}^M$ be a self similar grid on $X$ and 
\[
 \cE'_\cG := \setof{ e : e \text{ is a cell of } C_i \text{ for some } i = 0, \ldots , M}.
 \]
 Then 
 \[
 \cE_\cG : = \setof{ e \in  \cE'_\cG :  \text{ either } e\cap e' = \emptyset \text{ or } e \subseteq e' \text{ for every } e' \in \cE'_{\cG}}
 \]
is a CW structure on $X$ generated by $\cG$.
\end{lemma}

\begin{lemma}
Let $\cB$ be a  collection of building blocks of $X$ and   $\setof{\cG_i}_{i = -1}^{m+1}$ be a collection of self similar grids on $X$ such that $\cG_i \subseteq \cB$ for $i \in \setof{-1,\ldots, m+1}$. Then the CW structures  $\setof{\cE_{\cG_i}}_{i=-1}^{m+1}$,   generated by the grids $\setof{\cG_i}_{i = -1}^{m+1}$, are commensurable.
\end{lemma}

In the second part of this section we analyze the boundary operator associated with the CW structure $\cE_\cG$. 
Our goal is to express this operator using the incidence numbers given  by the basis   $\Lambda$ corresponding to $\cE$.  
We start by building a basis for the chain complex  $C_*(X)$ induced by the cellular structure $\cE_\cG$.
Every cell  $e \in \cE_\cG$ belongs to some $C_i \in \cG$. 
We associate this cell with one of these (arbitrary but fixed) building blocks  $C_i$. 
Let $h_i$ be the homeomorphism from $X$ to $C_i$. Then there is a cell $e'\in \cE$ such that $e = h_i(e')$.
Now the characteristic map for $e$ is given by $h_i \circ \Phi_{e'}$  where $\Phi_{e'}$ is the characteristic map for $e'$.  
The map $h_i$ induces a chain isomorphism between the chain complexes, $(h_i)_*:C_*(X) \to C_*(C_i)$.  
Fix $\lambda \in \Lambda$ to be the basis element associated with the cell $e' \in \cE$, and $\ell_i: C_i \to X$ the inclusion map. 
We choose $( \ell_i \circ h_i)_* (\lambda)$ to be a basis element in $C_*(X)$ with the cellular structure $\cE_{\cG}$, corresponding to the geometric cell $e$.
We denote $(\ell_i \circ h_i)_*(\lambda)$ by $(i,\lambda)$ and the corresponding basis by $\Lambda_\cG$. 
As before $|i,\lambda|$ is the topological cell associated  with $(i,\lambda)$.

To define a boundary operator in the basis $\Lambda_\cG$ using the incidence numbers for $\Lambda$, we need to understand the relationship between the cells belonging to different building blocks. Namely, we need a method for comparing the orientations of two intersecting cells with the same dimension, and we do so using their local homology groups. 
 
Let us consider  cells $e_i = |i,\lambda| \subset  C_i$ and  $e_j = |j,\mu| \subset  C_j$ such that $n = \dim(e_i)  = \dim(e_j)$ and $e_i \cap e_j \neq \emptyset$. 
Without loss of generality we can suppose that $e_i \subseteq e_j$ and select a point $ q \in e_i$.
Since $\bd(e_i)$ and $\bd(e_j)$ are deformation retracts of $\cl(e_i) \backslash q$ and $\cl(e_j) \backslash q$ respectively, 
then it follows that the maps induced by inclusion
\begin{eqnarray*}
r_*:H_*( \cl(e_i) ,  \bd(e_i))  \to H_*( \cl(e_i) ,  \cl(e_i) \backslash q) \\
s_*:H_*( \cl(e_j) ,  \bd(e_j))  \to H_*( \cl(e_j) ,  \cl(e_j) \backslash q) 
\end{eqnarray*}
are isomorphisms. 
By excision, the map $t_*: H_*( \cl(e_i) ,  \cl(e_i) \backslash q) \to H_*( \cl(e_j) ,  \cl(e_j) \backslash q) $ induced by inclusion is also an isomorphism. 
Thereby we obtain the following commutative diagram where all of the maps are isomorphisms: 
\begin{center}
\begin{tikzcd}
& H_*( \cl(e_i), \bd( e_i) ) \arrow{r}[swap]{r_*} & H_*( \cl(e_i) ,  \cl(e_i) \backslash q) \arrow{dd}[swap]{t_*}  \\
H_*( B^{n}, S^{n-1} ) \arrow{dr}[swap]{(h_j \circ \Phi_{|\mu|})_*} \arrow{ur}{(h_i \circ \Phi_{|\lambda|})_*} & \\
& H_*( \cl(e_j), \bd(e_j)  ) \arrow{r}{s_*} & H_*( \cl(e_j) ,  \cl(e_j) \backslash q)   
\end{tikzcd}
\end{center}
Hence for a fixed generator $a \in H_*(B^n,S^{n-1})$ then 
\begin{equation}
r_*^{-1} \circ t_*^{-1} \circ s_* \circ (h_j \circ \Phi_{|\mu|})_*  (a)   = \setof{(i,\lambda):(j,\mu)}  (h_i \circ \Phi_{|\lambda|})_*  (a) 
\label{eqn:Orientation}
\end{equation}
where $\setof{(i,\lambda):(j,\mu)}  $ is a unit. 
If we are working with $\Z$ coefficients then $\setof{(i,\lambda):(j,\mu)} = \pm 1 $ and  we say that $e_i$ and $e_j$ have the same orientation if  $\setof{(i,\lambda):(j,\mu)}  = 1$ and the opposite otherwise.

 First we  evaluate a boundary for some  cells in the complex  shown in Figure~\ref{fig:DyadicGrid}. The boundary of the two dimensional cell in the top left square $C$ consists of five one dimensional cells. Three of them belong to $C$. However the bottom edge was replaced by two smaller edges. Orientation of the removed edge agrees with the smaller edges. Hence we can simply replace this edge by those smaller edges. For the bottom right square  the problem is a little more complicated. The missing left edge is  divided into two smaller edges but orientation of the bottom smaller edge is opposite. Therefore we need to reverse  orientation of this cell by multiplying it by $-1$. The following lemma formalizes this procedure.
\begin{lemma}
Let $(i,\lambda) \in \Lambda^n_\cG$. Then 

\begin{equation}
\partial_n(i,\lambda) = \sum_{\mu \in \Lambda^{n-1}} [\lambda : \mu ] \sum_{(j,\tau) \in \cN(i,\mu)} \setof{(i,\mu):(j,\tau)}(j,\tau),  
\label{eqn:boundary}
\end{equation}
where $ \cN(i,\mu) = \setof{(j,\tau) \in \Lambda^{n-1}_\cG : |i,\mu| \cap |j,\tau| \neq \emptyset}$.
\end{lemma}

\begin{proof}

First we consider the boundary of  $(i,\lambda)$ inside  $C_i$ equipped   with the basis $\setof{ (i,\lambda) : \lambda \in \Lambda }$. 
This boundary can be written as $\sum_{\mu \in \Lambda^{n-1}} [\lambda : \mu ] (i,\mu)$. 
However this is not necessarily the boundary of $(i,\lambda)$ in the basis $\Lambda_\cG$.
This is because some elements  $(i,\mu)$ might  be represented by basis elements $(j,\tau)$ corresponding to a different building block,  
 or they can be broken up to smaller pieces also from different building blocks. 
The set $\cN(i,\mu) \subset \Lambda_\cG$ is a representation of $(i,\mu)$ in the basis $\Lambda_\cG$. 
 Geometrically, this means that $ \cl(|i,\mu|) = \bigcup_{(j,\tau)\in \cN (i , \mu)} \cl(|j,\tau|)$. 
To obtain   $\partial_n(i,\lambda)$ we have to replace each $(i,\mu)$ by the  elements in $\cN(i,\mu)$. 
In order to accommodate possible differences in  orientation, the sign of the incidence number $[\lambda : \mu]$ has to be changed if $(j,\tau) \in \cN(i,\mu)$ and $(i,\mu)$ have  opposite orientation. Therefore each term contained  in the second sum is multiplied by $\setof{(i,\mu), (j,\tau)}$.
\end{proof}

%%%%%%%%%%%%%%%%%%%%%%%%%%%%%%%%%%%%%%%%%%%%%%%%%%%%%%%%%%%%%%%%%%
\section{ Construction of compatible  CW structures on a unit cube}
%%%%%%%%%%%%%%%%%%%%%%%%%%%%%%%%%%%%%%%%%%%%%%%%%%%%%%%%%%%%%%%%%%
\label{sec::CompatibleCW}

By now the results presented in this paper have been mostly theoretical. 
In this section we restrict our attention to functions $f\in C^1([0,1]^N,\R)$ and present an algorithm for constructing  a  CW structure  compatible with  the sub-level set $X_t =  f^{-1}((-\infty, t])$.
It is complicated to check the conditions in Definition~\ref{def::CompatibleCW} algorithmically, especially Condition (3). 
Instead we introduce analytical conditions for dyadic cubes that force their cells to satisfy the conditions in Definition~\ref{def::CompatibleCW}. 
These conditions can be checked using interval arithmetic, and the algorithm can be carried out by a computer.  
We divide this section in two parts. In the first part we present the conditions posed on the dyadic cubes. The second part describes the algorithm for constructing  a  compatible  CW structures on a unit cube.

\subsection{Verified Cubes}
The dyadic grid $\cG$ is said to be $(f,t)$-\emph{verified} if every dyadic cube $C \in \cG$ is $(f,t)$-verified (cf. Definition \ref{def::IntroAnalyticVerified}). 
We will show that if $\cG$ is $(f,t)$-verified, then the associated CW complex $ \cE_{\cG}$ is compatible with $X_t$.   
Note that every cell in $\cE_{\cG}$ belongs to at least one  dyadic cube $C \in \cG$.  
Hence it suffices to show that every cell in a $(f,t)$-verified cube $ C \in \cG$ satisfies one of the three conditions in Definition~\ref{def::CompatibleCW}.  
The following lemma allows us to prove this for just the top dimensional cell.

\begin{lemma}
If a cube $C$ is $(f,t)$-verified, then for every cell $e \subseteq C$, the cube $\cl(e)$ is $(f,t)$-verified.
\label{prop::VerifiedInheritence}
\end{lemma}
\begin{proof}
Fix vectors $x_{l_1}, \dots , x_{l_n}$ such that $ 0 \notin \frac{ \partial f }{\partial x_{l_i}}(C)$ for all $ i \leq n$, with $0 \leq n \leq N$.  
Select any cell $ e \subseteq C$. 
There exists a $k \leq n$ such that if we reorder the sequence $ \{ l_i \}$, then $ 0 \notin \frac{ \partial f }{\partial x_{l_i}}(\cl(e))$ for all $ i \leq k$ and $ \cl(e)$ is orthogonal to $x_{l_i}$ for any $ i >k$. 
Select any $ \dim(e)-k$ dimensional cell $ e' \subseteq \cl(e)$ such that $ e'$ is orthogonal to the vectors $ x_{l_1}, \dots , x_{l_k}$. 
Since $C$ is a cube, then there exists a $N-n$ cell $e'' \subseteq C$ such that $ e' \subseteq e''$ and $e''$ is orthogonal to the vectors $ x_{l_1}, \dots , e_{l_n}$. 
Since $C$ is $(f,t)$-verified, then either $ f( \cl(e'') ) >t$ or $ f(\cl(e'')) \leq t$. 
Thereby either $ f( \cl(e') ) >t$ or $ f(\cl(e')) \leq t$, thus showing that $ \cl(e)$ is $(f,t)$-verified. 
\end{proof}

\noindent There are three special cases of a cube $C$ being $(f,t)$-verified corresponding to the cases when $n =0$ and $n=N$:
\begin{enumerate}
	\item $f(C) > t $
	\item $f(C) \leq  t$ 
	\item $0 \notin \frac{\partial f}{\partial x_i}(C)$ for $ 1 \leq i \leq N$. 
\end{enumerate}
If a $(f,t)$-verified cube  $C$ satisfies one of Conditions (1) or (2) above, then all of its cells satisfy Condition (1) or (2) in Definition~\ref{def::CompatibleCW} respectively.  
On the other hand, if a cube $C$ satisfies neither of Conditions (1) or (2) above, then we must construct a deformation retraction required by Condition (3) in Definition~\ref{def::CompatibleCW}.
Using rescaling we can further restrict our attention to a unit cube.

\begin{figure}[tb]
\subfigure{\includegraphics[width = .4\textwidth]{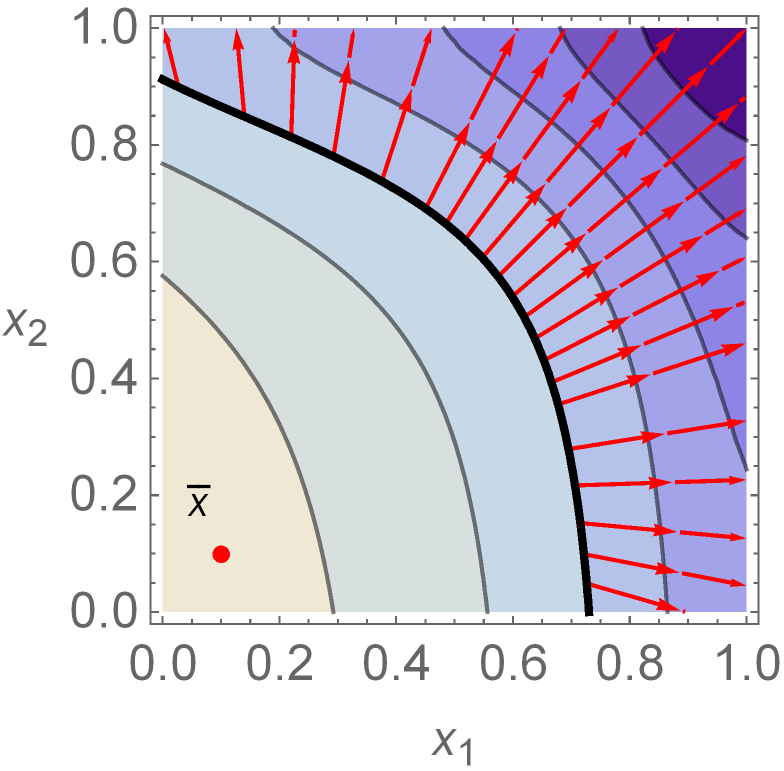}}
\subfigure{\includegraphics[width = .39\textwidth]{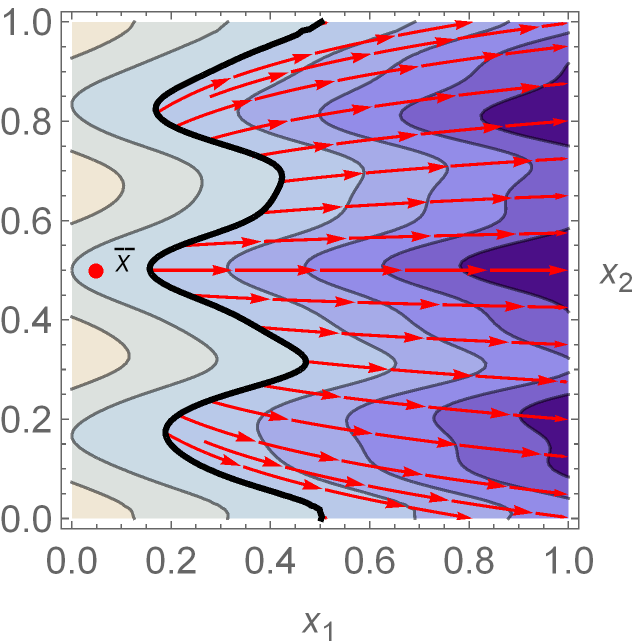}}
\caption{(a) A deformation retract of a sub-level set to the boundary when the cube satisfies Condition (3).  (b)  A deformation retract of a sub-level set to the boundary for a generic $(f,t)$-verified cube. }
\label{fig::Homotopy}
\end{figure}

Before constructing the deformation retraction in full generality, we demonstrate the basic ideas on two examples shown in Figure~\ref {fig::Homotopy}. 
The unit cube depicted in  Figure~\ref{fig::Homotopy}(a) satisfies Condition (3) above. 
In this case both partial derivatives are negative. 
For $\bar{\bx} \in \inter ([0,1]^2  \setminus X_t)$
we define a vector field $\xi(\bx) = \bx - \bar{\bx}$. 
The set $X_t$ can be retracted onto $X_t \cap \partial [0,1]^2$ along the flow lines, indicated by the arrows in   Figure~\ref {fig::Homotopy}(a), of the flow generated by $\xi(\bx)$.  To ensure that the function $f \circ h(\bx,s)$ is non-increasing in $s$ we have to choose $\bar{\bx}$ sufficiently close to the origin.

For the example shown in Figure~\ref {fig::Homotopy}(b) the partial derivative $\frac{\partial f}{\partial x_1} ([0,1]^2) < 0$ and $\frac{\partial f}{\partial x_2}$ is both positive and negative on $[0,1]^2$. 
In this case the cells $e_1 = \{0 \} \times (0,1)$ and $e_2 = \{1\} \times (0,1)$ are orthogonal to the coordinate vector $x_1$, and satisfy $f(\cl(e_1)) >t$ and $f(\cl(e_2)) \leq t$.
Hence cube $[0,1]^2$ is $(f,t)$-verified. 
In this case we define the vector  field  $\xi(x_1,x_2) = (x_1 - \bar{x}_1, \varepsilon (x _2 - \frac{1}{2}))$ for some point $(\bar{x}_1,\frac{1}{2}) \in \inter (X_t \setminus [0,1]^2)$ and $\varepsilon$ small enough. Again we retract the set $X_t$ along the flow lines indicated by the arrows.
The first coordinate of the vector field $\xi$ is similar to the one used in the previous example. 
Part of the set $X_t \cap \partial [0,1]^2$ is contained inside  the  horizontal edges of the cube. The second  coordinate of the vector field $\xi$ ensures that some flow lines reach this part of the set  $X_t \cap \partial [0,1]^2$.
The following two lemmas formalize the above discussion.

\begin{lemma}
If $0 \notin \frac{\partial f}{\partial x_i}([0,1]^N)$ for $ 1 \leq i \leq N$ and $f$ passes through $t$, then there exists a deformation retraction $h$ from $X_t $ onto $X_t \cap \partial [0,1]^N$. Moreover  $f \circ h(\bx,s) $ is non-increasing in $s$. 
\label{lem::C3}
\end{lemma}

\begin{proof}

Without loss of generality we  can suppose that $\frac{\partial f}{\partial x_i}(\bx) < 0$ for all $\bx \in [0,1]^N$ and $ 1 \leq i \leq N$.
Thereby the value of $f$ at the origin has to be  greater than $t$. To prove the existence of a deformation retraction $h : X_t \to X_t \cap \partial [0,1]^N$ we use the Wa\.{z}ewski principal. We  will show that $X_t$ is a Wa\.{z}ewski set for  a flow $\varphi$ corresponding to an appropriate  linear vector field $\xi(\bx)$ on $\R^N$.

The vector field is given by $\xi(\bx) = \bx - \bar{\bx}$ where $\bar{\bx} \in \inter(  [0,1]^N \setminus X_t)$ such that  $(\bx-\bar{\bx} ) \cdot \nabla f (\bx) < 0$ for  all $\bx \in X_t$. 
To show that there exists such  $\bar{\bx}$  we analyze the continuous function $\bx \cdot \nabla f (\bx)$.  
For every $\bx \in [0,1]^N$ the function  $\bx \cdot \nabla f (\bx) \leq 0$ and it  attains zero only at the origin. 

So there exists $\varepsilon > 0$ such that $\bx \cdot \nabla f (\bx) \leq - \varepsilon$ on $X_t$. Moreover for $\delta > 0$ sufficiently small $\bx \cdot \nabla f (\bx) >  \frac{- \varepsilon}{2}$ for all $\bx \in B_\delta := \setof{\by \in [0,1]^N : \| \by \| < \delta}$. 
We choose  $\bar{\bx}$ to be an arbitrary but fixed point in  $B_\delta \cap \inter([0,1]^N)$. 
Clearly $\bar{\bx} \not\in X_t$ and $(\bx-\bar{\bx} ) \cdot \nabla f (\bx) < 0$ for all $\bx \in X_t$.

Let $\varphi$ be a flow given by the vector  field $\xi(\bx)$. We define the exit set and the immediate exit set for $X_t$ as follows:
\begin{eqnarray*}
	W^0 &=& \setof{ \bx \in X_t : \exists \, s >0 \mbox{ such that } \varphi(\bx,s) \notin X_t  }, \\
	W^- &=& \setof{ \bx \in X_t :  \forall \, s >0 \mbox{ the set }  \varphi(\bx,[0,s)) \not\subset X_t }. 
\end{eqnarray*}

The  unstable fixed point of the linear system $\bx' = \xi(\bx)$ is not  in $X_t$. Therefore all the orbits have to leave $X_t$ and $W^0 = X_t$.
By the chain rule $\frac{d}{d s} f \circ \varphi(\bx,s) = (\bx-\bar{\bx} ) \cdot \nabla f (\bx)$ which is negative  for all $\bx \in X_t$ and $f \circ \varphi(\bx,s)$ is strictly decreasing with respect to $s$ for all  $\bx \in X_t$. 
Hence for every $\bx \in X_t \cap (0,1)^N$  there exists $ s_\bx >0$ such that $ \varphi(\bx,[0,s_\bx)) \subseteq X_t$ and the immediate exit set satisfies $W^- \subseteq X_t \cap \partial [0,1]^N$. 
The point $\bar{\bx}$  is in $\inter( [0,1]^N)$.  Thus  for every $\bx \in \partial [0,1]^N$ the vector field $\xi(\bx)$ points out of the unit cube and  $ \varphi(\bx,s) \notin [0,1]^N$ for $s > 0$. Thereby $ W^- = X_t \cap \partial [0,1]^N$. 

The sets $X_t$ and $X_t \cap \partial [0,1]^N$ are closed so it follows that  $X_t$ is a Wa\.{z}ewski set.  
By the Wa\.{z}ewski principal there exists a deformation retraction $h(\bx,s)$ from $W^0 =  X_t$ onto $W^- =  X_t \cap \partial [0,1]^N$.  The set $X_t$ is retracted along the flow lines of $\varphi$. So $f \circ h(\bx,s) $ is non-increasing in $s$.

\end{proof}

\begin{lemma}
If $[0,1]^N$ is $(f,t)$-verified and $f$ passes through $t$, then  there exists a deformation retraction $h$ from $X_t$ onto $X_t \cap \partial [0,1]^N $. Moreover $f \circ h(\bx,s) $ is non-increasing in $s$. 
\label{lem::C4}
\end{lemma}

\begin{proof}

Without loss of generality we can suppose   that $\frac{\partial f}{\partial x_i}(\bx) < 0$ for all $\bx \in [0,1]^N$ and $1 \leq i \leq k$.  We denote by $\pi_A$  projection on the first $k$ coordinates and by $\pi_B$  projection on the last $N-k$ coordinates. 
The cube $B = \pi_{B}([0,1]^N)$ is the closure of a cell in $[0,1]^N$ which is orthogonal to the vectors $ x_{1}, \dots , x_{k}$, so either $ f(B) >t$ or $ f(B) \leq t$.  
Since $ t \in f([0,1]^N)$ and $ \frac{\partial f}{ \partial x_i}([0,1]^N) <0$ for $1 \leq i \leq k$ then it follows that $ f(B) >t$. 

We apply the Wa\.{z}ewski principal  to a linear flow $\varphi$ radiating from some point $\bar{\bx}$.   By the same argument as in the previous proof one can show that for  some  $\varepsilon > 0$ there exists $\bar{\bx} = (\bar{x}_1, \ldots, \bar{x}_k, {1/2}, \ldots, {1/2}) \in \inter ( [0,1]^N \setminus X_t)$  such that    $\pi_A(\bar{\bx} - \bx) \cdot \nabla f(\bx) < -\varepsilon$ for all  $\bx \in X_t$.  
We now define the following vector field 
\[
\xi(\bx) = \pi_A(\bx - \bar{\bx}) + \frac{\varepsilon}{1 + \eta} \pi_B(\bx -\bar{\bx})
\]
where $\eta  =  \sup_{\bx \in X_t} | \pi_B (\bx-\bar{\bx}) \cdot \nabla f(\bx) |$. The value $\eta$ has to be finite because it is  supremum of  a continuous function on the compact set $X_t$. 
Define $\varphi$ to be the flow generated by the vector field $ \xi$.

%Without loss of generality, suppose that $\frac{\partial f}{\partial x_i}(p) < 0$ for all $p \in C$ and $ 1 \leq i \leq k$.
%Let $ \pi_A:C \to A$ and $ \pi_B:C \to B$ be projection maps. 
%Since $t \in f(C)$ then $f( 0 \times B) > t$.  
%By the proof to the previous lemma, there exists a vector field $ \xi_1:\R^k \to \R^k \times 0^{n-k}$ and an $ \epsilon>0$ such that $\xi_1 \circ  \pi_A (x) \cdot \nabla f (x) < -\epsilon$ for all $ x \in X_t \cap C = X_t$.  
%Furthermore, if $ \varphi_1$ is the associated flow, then $\varphi_1(x, s) \notin C$  for all $x \in X_t \cap (\partial A \times B)$ and $s >0$.
%Let $ b_{1/2} \in B$ denote the center point of the cube $B$ and define 
%\[
%\eta  =  \sup_{x \in X_t} | (x-b_{1/2}) \cdot \nabla f(x) | 
%\] 
%Since $X_t$ is compact then $\eta$ is finite, and so we may define a nonzero radial vector field $ \xi_2: \R^{n-k} \to 0^k \times \R^{n-k}$ by 
%\[
%\xi_2(x) = \frac{\epsilon}{1 + \eta} (x - b_{1/2}).
%\]
%It follows that $ |\xi_2 \circ  \pi_B(x) \cdot \nabla f(x) | < \epsilon$ for all $ x \in C$.  
%We now define a vector field $ \xi $ on $\R^n$ by
%\[
%\xi(x) = \xi_1  \circ \pi_A (x) + \xi_2 \circ \pi_B (x)
%\]
%and let $ \varphi $ be its associated flow. 
%It then follows for all $x \in X_t$ that 

The function $f$ is non-increasing along the flow lines of $\varphi$ as corroborated by the following calculation

\begin{eqnarray*}
\frac{d}{dt} f \circ \varphi(\bx,s) &=& \left(\pi_A(\bx - \bar{\bx}) + \frac{\varepsilon}{1 + \eta} \pi_B(\bx -\bar{\bx})\right) \cdot \nabla f(\varphi(\bx,s)) \\
&<& -\epsilon + \epsilon \frac{ \eta}{1 + \eta} < 0 .
\end{eqnarray*}
Applying the Wa\.{z}ewski principal  to the same sets $W^0$ and $W^-$ as in the proof of the previous lemma guarantees that there exists a deformation retraction $h$ from $X_t$ onto $X_t \cap \partial [0,1]^N $ and $f \circ h(\bx,s) $ is non-increasing in $s$.

%Define sets 
%\begin{eqnarray*}
%	W^0 &=& \{ x \in X_t | \exists \, s >0 \mbox{ such that } \varphi(x,s) \notin X_t  \} \\
%	W^- &=& \{ x \in X_t | \forall \, s >0 \mbox{ then }  \varphi(x,[0,s)) \not\subset X_t   \} 
%\end{eqnarray*}
%Since the orbit of every point under $ \varphi$ is unbounded, then clearly $W^0 = X_t$.  
%Since $f \circ \varphi(x,s)$ is strictly decreasing with respect to $t$, then for all $x \in X_t$ there is some $ \epsilon >0$ such that $ \varphi(x,[0,\epsilon]) \subset X_t$, 
%and thereby $ W^- \subseteq X_t \cap \partial C$. 
%Since $\xi$ is radiating out from an interior point of $C$, then for all $x \in \partial C$  and $ s>0$ we have $ \varphi(x,s) \notin C$. 
%Thereby $ W^- = X_t \cap \partial C$. 
%Then since both $X_t \cap C$ and $X_t \cap \partial C$ are closed, then $X_t \cap C$ is a Wa\.{z}ewski set.  
%By the Wa\.{z}ewski principal $X_t \cap \partial C$ is a strong deformation retract of $X_t \cap C$ which concludes the proof.  

\end{proof}

Finally we are in the position to show that an $(f,t)$-verified grid $\cG$ induces a CW structure $\cE_\cG$ compatible with $X_t$.

\begin{theorem}
Let  $f\in C^1([0,1]^N,\R)$. If the  dyadic grid $\cG$ on  $[0,1]^N$  is $(f,t)$-verified, then the CW structure $\cE_\cG$  is compatible  with $X_t$.
\label{thm::VerifiedMeansCompatible}
\end{theorem}

\begin{proof}
If $e$ is a cell in $ \cE_{\cG}$, then there exists some $(f,t)$-verified  cube $C \in \cG$ which contains $e$. 
By Lemma \ref{prop::VerifiedInheritence} it follows that $\cl(e)$ is $(f,t)$-verified.  
If either $f(\cl(e)) >t$ or $f( \cl( e)) \leq t$, then $\cl(e)$ will satisfy Condition (1) or (2) in Definition~\ref{def::CompatibleCW}, respectively.  
Otherwise $f$ passes through $t$, and so by Lemma \ref{lem::C4} there exists a deformation retraction $h$ of  the set $ X_t \cap \cl(e)$ onto $X_t \cap \bd(e)$. 
Moreover  the function $f \circ h(\bx, s)$ is non-increasing in $s$. 
Thereby $\cE_\cG$ is compatible  with $X_t$.
\end{proof}

\begin{remark}
The fact that the function  $f \circ h(\bx, s)$ is non-increasing in $s$ implies that the sets $\cX_t$ obtained using the dyadic grid can be used to produce a cellular filtration (see Theorem \ref{prop::ApproximateFiltration}).
\end{remark}

\subsection{Algorithm for Constructing a Verified Grid}

Recall that a dyadic grid $\cG$ is  $(f,t)$-verified if every cube $C \in \cG$ is  $(f,t)$-verified. To construct a verified grid we start with a grid $\cG$ containing a single cube $[0,1]^N$. Then we iteratively refine the cubes  $C \in \cG$ that are not  $(f,t)$-verified.
If $C$ is not $(f,t)$-verified, then we replace it by $2^N$ dyadic cubes inside $C$ whose edge length equals to half of the edge length of $C$. 
We call those smaller cubes the offspring of $C$.
If  the following algorithm terminates,  then the returned grid $\cG$ is $(f,t)$-verified.\\

\noindent $\cG$ = \textbf{\texttt{ConstructVerifiedGrid}}$(f, t)$  
\begin{algorithmic}
\STATE $\cG = \setof{[0,1]^N}$
\WHILE{there exists a cube   $C \in \cG$ which is not $(f,t)$-verified}
\STATE Remove $C$ from $\cG$
\STATE Add the offspring of $C$  to $\cG$
\ENDWHILE
\RETURN $\cG$
\end{algorithmic}

%%%%%%%%%%%%%%%%%%%%%%%%%%%%%%%%%%%%%%%%%%%%%%%%%%%%%%%%%%%%%%%%%%
\section{Data Structures and Algorithms}
%%%%%%%%%%%%%%%%%%%%%%%%%%%%%%%%%%%%%%%%%%%%%%%%%%%%%%%%%%%%%%%%%%
\label{sec:DataAndAlgorithms}

\begin{figure}[tb]
\subfigure{\includegraphics[height = .22\textwidth]{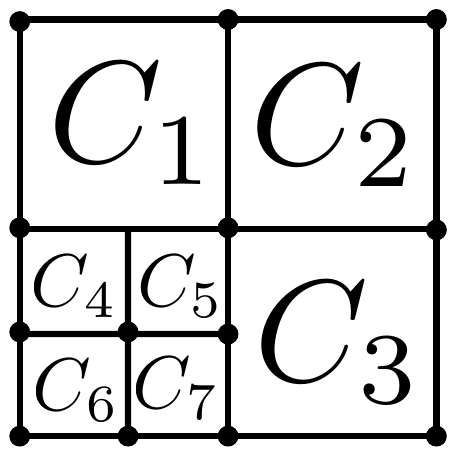}}
\subfigure{\includegraphics[height = .22\textwidth]{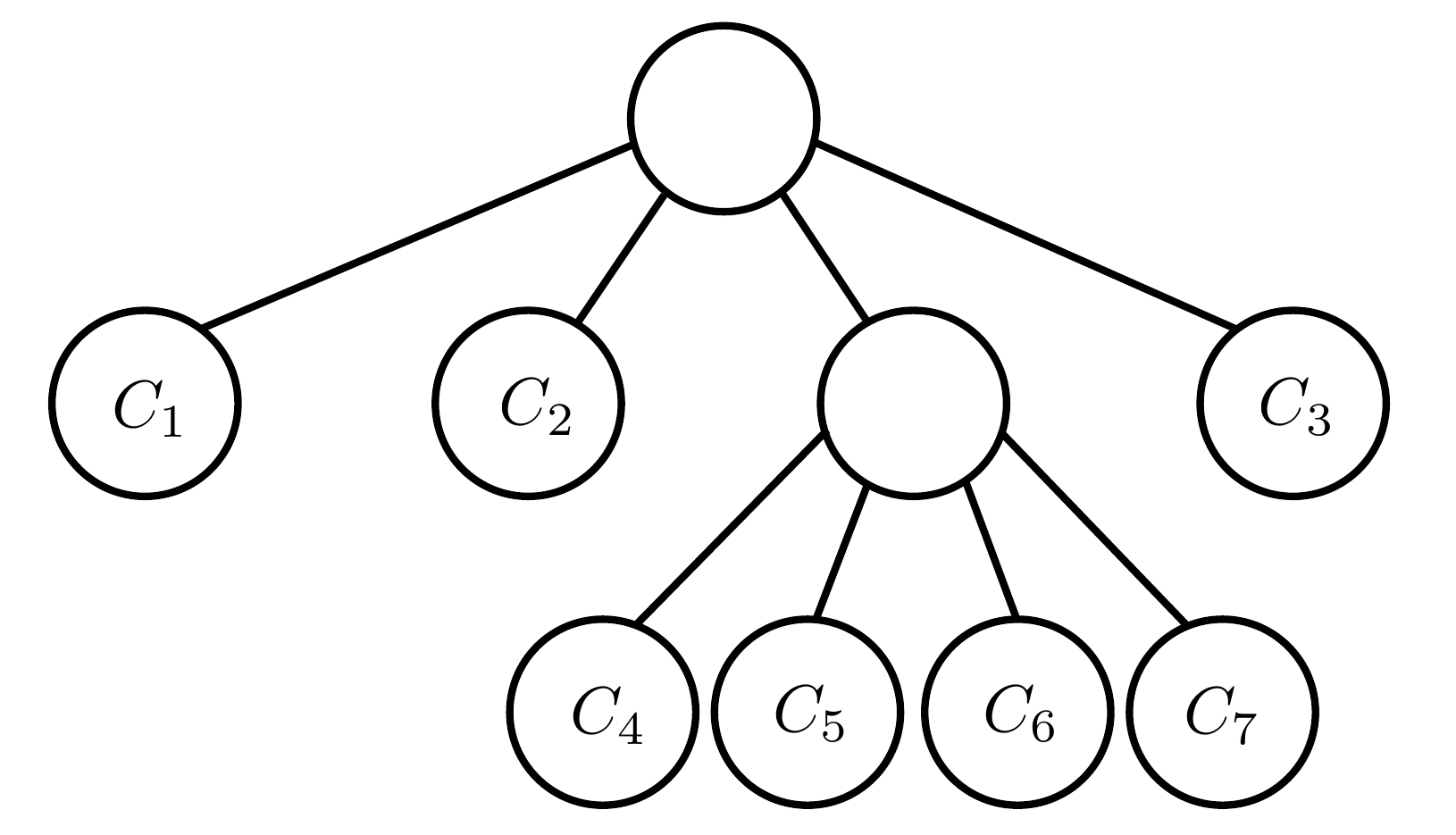}}
\subfigure{\includegraphics[height = .22\textwidth]{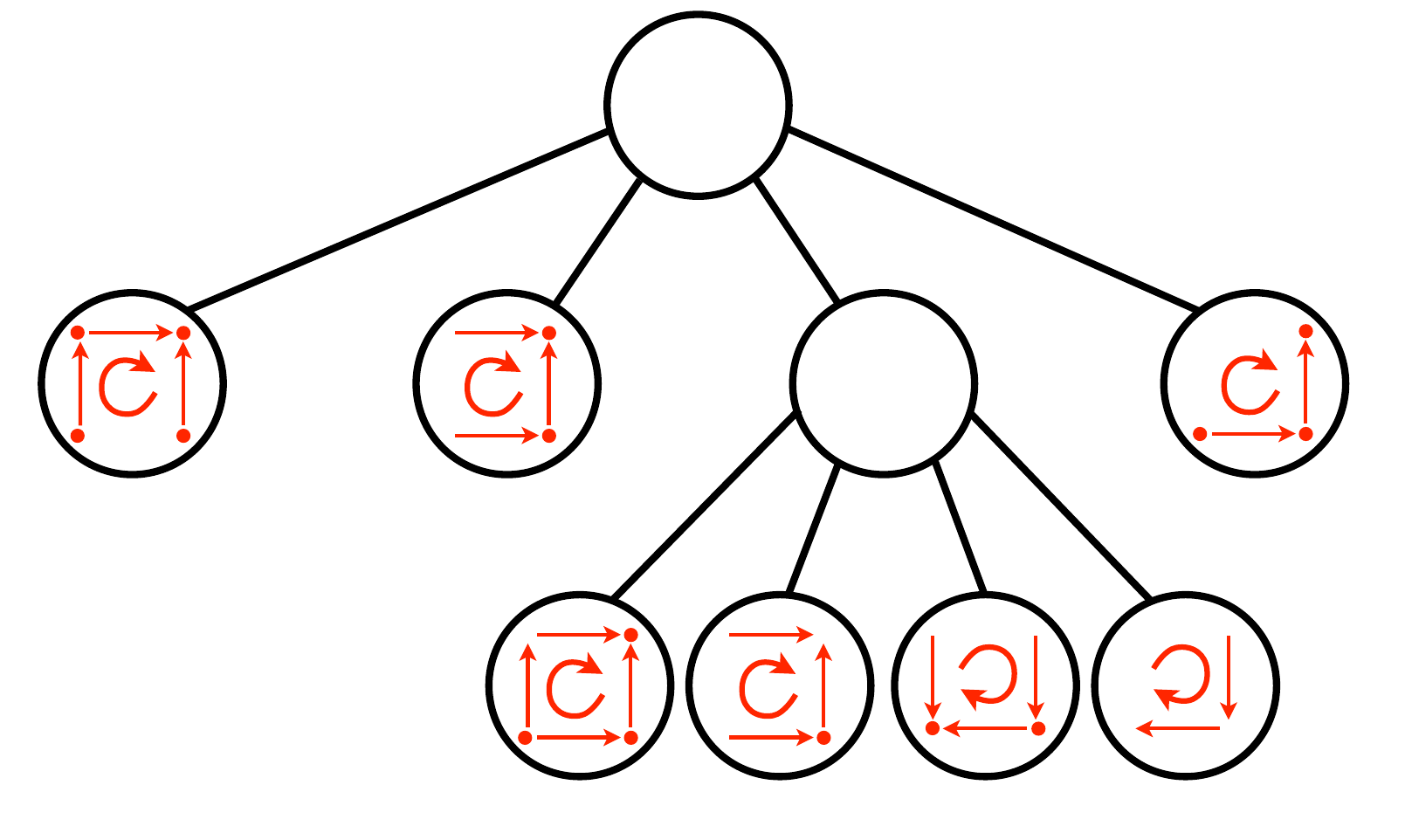}}
\caption{(a)  A  dyadic  cubical grid $\cG$ with seven  building blocks. (b) Tree representation of $\cG$. A building block $C_i \in \cG$ corresponds to the leaf with  same label. (c)  Basis $\Lambda_\cG$ chosen in Figure~\ref{fig:DyadicGrid}(c) is  stored in the leaves of the tree corresponding to $\cG$.}
\label{fig::Tree}
\end{figure}

In order to compute the homology of the sub-level set $X_t \subseteq X$ we need to develop data structures for storing the basis $C_*(\cX_t)$ and the boundary operator. In Section~\ref{sec:Representation} we introduce a representation of  the basis $\Lambda_\cG$ associated with the grid $\cG$. A special structure of  self similar grid  is exploited to devise a memory efficient data structure. We do not store the boundary operator for the basis $\Lambda_\cG$ which is usually  a large (sparse) matrix and requires a lot of memory to be stored. 
Only the incidence numbers for the coarse CW structure $\cE_{\cG}$ need to be stored. In the case of a unit cube the incidence numbers can be computed and the memory requirements are further reduced.

In Section~\ref{sec:BoundaryOperator}, the boundary operator $\partial_n \colon \Lambda^n_\cG \to \Lambda^{n-1}_\cG$ is computed using Formula~(\ref{eqn:boundary}).  
Evaluating the terms in (\ref{eqn:boundary}) involves manipulating geometric objects. 
This might be complicated for a general self similar grid. 
However for  a dyadic grid it can be performed using a special representation of the cells, as explained in  Section~\ref{sec:DyadicGrid}.

%%%%%%%%%%%%%%%%%%%%%%%%%%%%%%%%%%%%%%%%%%%%%%%%%%%%%%%%%%%%%%%%%%%%%
\subsection{Representation of $\Lambda_\cG$}
%%%%%%%%%%%%%%%%%%%%%%%%%%%%%%%%%%%%%%%%%%%%%%%%%%%%%%%%%%%%%%%%%%%%%
\label{sec:Representation}

A collection of  building blocks $\cB$, given by Definition~\ref{def:BuildingBlocks}, can be partially ordered by  inclusion. For $C_i,C_j \in \cB$ we say that $C_i$ precedes $C_j$ and write   $C_i \prec C_j $ if $C_j \subseteq C_i$. 
The fact that $X \prec C$ for all $C \in \cB$ implies that the graph representing the partial ordering is a tree $T_\cB$ with the root corresponding to the set $X$.   Therefore a self similar grid $\cG = \setof{C_i}_{i=0}^M$ can be represented by a subtree $T_\cG$ of $T_\cB$   containing  the vertices corresponding to the building blocks $C_i\in \cG$ and their predecessors. 

Figure~\ref{fig::Tree} shows an example of a  self similar grid with seven building blocks and its tree representation $T_\cG$. There is a one to one correspondence between the building blocks  $C_i \in \cG$ and the leaves of the tree $T_\cG$. The nodes of the tree correspond to the building blocks that were refined in the process of creating the grid.  In our example the root corresponds to the unit square and the unlabeled node in the middle layer to the subdivided cube at the bottom left corner. 

\begin{definition}
For every  node (leaf)  $\bs$  in $T_\cG$ we denote the corresponding building block by $B(\bs)$. 
If $\bs$ is not a leaf, then $B(\bs) \not\in \cG$. 
For a leaf $\bs$ we define 
$\Lambda^n(\bs) := \setof{ (i,\lambda) \in \Lambda^n_\cG \text{ stored in the leaf } \bs}$.
\end{definition}

The basis $\Lambda_\cG$  is  stored in the leaves of the tree $T_\cG$.   
Figure~\ref{fig::Tree}(c) shows how the basis of the CW structure chosen in Figure~\ref{fig:DyadicGrid}(c) is stored. In general, the leaf corresponding to the building block $C_i \in \cG$ stores the basis elements $(i,\lambda) \in \Lambda_\cG$.

%%%%%%%%%%%%%%%%%%%%%%%%%%%%%%%%%%%%%%%%%%%%%%%%%%%%%%%%%%%%%%%%%%%%%
\subsection{Boundary Operator}
%%%%%%%%%%%%%%%%%%%%%%%%%%%%%%%%%%%%%%%%%%%%%%%%%%%%%%%%%%%%%%%%%%%%%
\label{sec:BoundaryOperator}
Evaluation of Formula (\ref{eqn:boundary})  requires construction of the set $\cN(i,\mu)$.
In this section we present a tree traversing algorithm for  constructing  $\cN(i,\mu)$. The algorithm consists of two parts. In the first part,  the algorithm moves up the tree. 
It stops  at the node $\bs$ such that $|i,\mu| \subseteq \inter(B(\bs))$. We will prove that  the elements of $\cN(i,\mu)$ are stored in the leaves that can be reached by descending down form  the node $\bs$. In the second part, the algorithm descends to the leaves that can contain  elements of $\cN(i,\mu)$ and retrieves  them. Before stating the algorithm  we recall some basic terminology.

\begin{definition}
Let $T$ be a tree. Suppose that $\bs,\bs'$ are two nodes of $T$ connected by an edge. We say that $\bs$ is a predecessor (child) of $\bs'$ if the path from $\bs$ to the root is shorter (longer) than the path from $\bs'$ to the root.
\end{definition}

\noindent $\cN$ = \textbf{\texttt{ConstructN}}$(i,\mu)$  
\begin{algorithmic}
\STATE Let $\bs$ be a leaf corresponding to $C_i \in \cG$.
\WHILE{ $|i,\mu| \not\subseteq \inter(B(\bs))$ and $\bs$ is not the root }
\STATE Set $\bs$ to the predecessor of $\bs$
\ENDWHILE

\STATE Put $\bs$ to the stack $\cH$

\WHILE{ $\cH \neq \emptyset$}

\STATE Remove one element  $\bs$ from  $\cH$.

\IF{$\bs$ is a leaf}

\FOR{ every $(j,\tau) \in \Lambda^{\dim(\mu)}(\bs)$ such that $|i,\mu| \cap |j,\tau| \neq \emptyset$ }
\STATE Put  $(j,\tau)$ to $\cN(i,\mu)$
\ENDFOR

\ELSE

\FOR{ every child $\bs'$ of $\bs$ such that $|i,\mu| \cap B(\bs') \neq \emptyset$}
\STATE Add $\bs'$ to  $\cH$
\ENDFOR

\ENDIF
\ENDWHILE

\RETURN $\cN$
\end{algorithmic}

The following lemma guarantees that the set $\cN$  retrieved by  \textbf{\texttt{ConstructN}}$(i,\mu)$  is equal to the set $\cN(i,\mu)$ given by~(\ref{eqn:boundary}).

\begin{lemma}
For every basis element $(i,\mu) \in \Lambda_\cG$ the algorithm \textbf{\texttt{ConstructN}}$(i,\mu)$  returns the set $\cN(i,\mu)$.
\end{lemma}

\begin{proof}
Let $\bs$ be the node  reached at the end of the ascending part (first while cycle) of the algorithm. 
We show that  every leaf containing an element of $\cN(i,\mu)$ can be reached by descending from the node $\bs$. If  $\bs$ is the root, then this is trivial. If $\bs$ is not the root, then $|i,\mu| \subseteq \inter(B(\bs))$. 
For every $(j,\tau) \in \cN(i,\mu)$ the cell $|j, \tau| \subseteq |i,\mu| \subseteq \inter (B(\bs))$. 
So $\inter(C_i) \cap \inter( B(\bs) ) \neq \emptyset$ for every building block $C_j \in \cG$, such that   $|j,\tau| \subseteq C_j$.
The interior of every building block is formed by a single $N$ dimensional cell.
Hence by Condition~(2) of Definition~\ref{def:BuildingBlocks} either  $\inter( C_j )\subseteq \inter (B(\bs))$  or $\inter( B(\bs) )\subseteq \inter (C_j )$. 
If $\inter( B(\bs) )\subseteq \inter (C_j )$, then  there exists a building block $C_k \in \cG$, corresponding to some leaf below $\bs$, such that $\inter(C_j) \cap \inter(C_k) \neq \emptyset$. 
This  violates Condition~(2) of Definition~\ref{def:SelfSimilarGrid}.  Therefore $C_j \subseteq B(\bs)$ and the  leaf, containing $(j,\tau)$ can be reached by descending from the node $\bs$.

In the descending part,  the algorithm reaches every leaf $\bs$ corresponding to a building block $C_j$ that intersects $|i,\mu|$. All the elements of $\Lambda^{\dim(\mu)}(\bs)$ which intersect $|i,\mu|$ are added to $\cN$. Hence $\cN = \cN(i,\mu)$.
\end{proof}

%%%%%%%%%%%%%%%%%%%%%%%%%%%%%%%%%%%%%%%%%%%%%%%%%%%%%%%%%%%%%%%%%%%%%
\subsection{Dyadic Grid}
%%%%%%%%%%%%%%%%%%%%%%%%%%%%%%%%%%%%%%%%%%%%%%%%%%%%%%%%%%%%%%%%%%%%%
\label{sec:DyadicGrid}

\begin{figure}[tb]
\subfigure{\includegraphics[height = .32\textwidth]{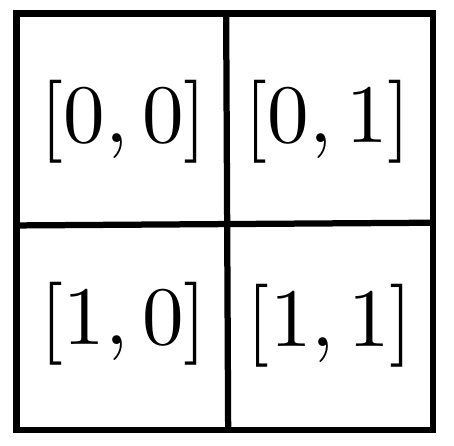}}
\subfigure{\includegraphics[height = .32\textwidth]{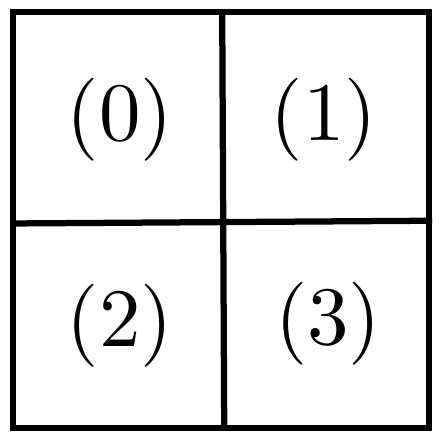}}
\subfigure{\includegraphics[height = .32\textwidth]{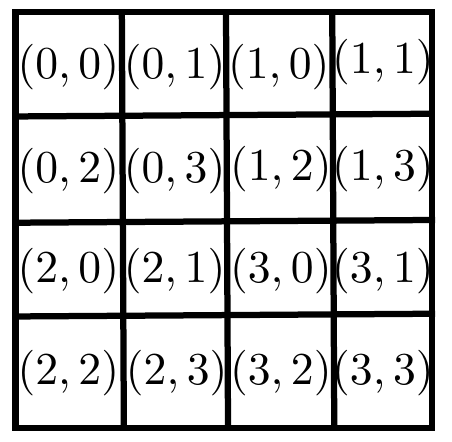}}
\caption{ (a) Binary (b) decimal indices for the dyadic cubes obtained by one subdivision of a unit square. (c) Binary indices for the dyadic cubes obtained by two subdivision of a unit square.  }
\label{fig:GridLabels}
\end{figure}

In this section we deal with implementation of the algorithm for a dyadic grid. First we label the dyadic cubes. 
Our labeling  schema is motivated by the refinement process applied to the  unit cube. During a single refinement the cube $[0,1]^N$  is   divided into $2^N$ cubes with edge lengths equal to $1/2$. Each of these cubes contains exactly one vertex of the original cube. Let $C$ be a cube containing the vertex $(s^0, s^1, \ldots, s^{N-1})$.
The coordinates of the vertex can be interpreted as a binary number. The number $s = \sum_{n=0}^{N-1} s^n 2^n$ is the index of the cube $C$. Note that  the index contains complete information about  position of the cube. Coordinates of the  smallest vertex (in the lexicographical order)  of $C$  are given by $x_i = \frac{s^i}{2}$, see Figure~\ref{fig:GridLabels}.

Further refinement of the cube $C$, considered above, produces another collection of $2^N$ dyadic cubes. Each of these cubes  contains exactly one vertex of $C$. Again, the position of each cube inside    $C$ is encoded by a binary number with $N$ digits.  Therefore every cube obtained by two consecutive refinements of $[0,1]^N$ is indexed by a sequence $(s,s')$, see Figure~\ref{fig:GridLabels}(c). In general a dyadic cube obtained by $n$ subdivisions is uniquely indexed by  $\bs = (s_1,s_2,\ldots, s_n)$ where $s_i \in  \setof{0,1,\ldots, 2^{N}-1}$. The position of the cube corresponding to   $\bs$ is given by 
\begin{equation}
C_{\bs} = \prod_{i = 0}^{N-1} \left[ \sum_{j=1}^n 2^{-j} s_j^i  , \sum_{j=1}^n 2^{-j} (1+ s_j^i )  \right],
\label{eqn:CubeCoordinates}
\end{equation}
where $s_j^i$ is the $i$-th binary digit of $s_j$.

 We choose a CW structure for the unit cube to be  the standard cellular structure on $[0,1]^N$. Every cell $e$ is a product of $N$ generalized intervals and can be uniquely indexed by $\lambda =(\lambda^0, \lambda^1, \ldots ,\lambda^{N-1}) \in \Z_3^N$. In particular the cell corresponding to the index $\lambda$ is $|\lambda| = I_0 \times I_1 \times \ldots \times I_{N-1}$ where 
\[
I_i = \begin{cases}
0 & \text{if $\lambda^i = 0$} ,\\
1 & \text{if $\lambda^i = 1$} ,\\
(0,1) & \text{if $\lambda^i = 2$.}
\end{cases}
\]
If the cell $e'$ is a translation of the cell $e$, then $\Phi_{e'} = T \circ \Phi_e$, where $T$ is the translation map.
The characteristic map of the cell $e$ whose closure  contains the origin can be chosen arbitrarily.  The orientation of the cells $e'$ that do not contain the origin in their closure is given by $\Phi_{e'}$.

To define  a CW structure on $C_\bs$ we use the homeomorphism $h_{\bs} \colon [0,1]^N \to C_\bs$ defined coordinate-wise by 
\begin{equation}
h_i(\bx) = 2^{-n} x_i + \sum_{j = 1}^n 2^{-j}s_j^i,
\label{eqn:H}
\end{equation}
where $h_{\bs} =(h_0(\bx), h_1(\bx), \ldots , h_{N-1}(\bx))$.
Our choice of orientation and homeomorphisms $h_\bs$ guarantees that two intersecting cells have the same orientation. So Formula~ (\ref{eqn:boundary}) reduces to 
\[
\partial_n(i,\lambda) = \sum_{\mu \in \Lambda^{n-1}} [\lambda : \mu ] \sum_{(j,\tau) \in \cN(i,\mu)} (j,\tau). 
\]
Implementation of the algorithm \textbf{\texttt{ConstructN}}$(i,\mu)$ requires evaluation of  intersections of the cells.  For the dyadic grid the cells are cubes and their intersection can be computed from the cell index $(\bs,\lambda)$ using (\ref{eqn:CubeCoordinates}) and  (\ref{eqn:H}). In practice we do not compute the intersections. Instead we formally manipulate  indices of the cells to obtain the index of the cell corresponding to the intersection.

%%%%%%%%%%%%%%%%%%%%%%%%%%%%%%%%%%%%%%%%%%%%%%%%%%%%%%%%%%%%%%%%%%
\section{Application}
%%%%%%%%%%%%%%%%%%%%%%%%%%%%%%%%%%%%%%%%%%%%%%%%%%%%%%%%%%%%%%%%%%

\subsection{ Error Estimates for Persistence Diagrams}
The algorithm we have described for computing persistence diagrams largely depends on constructing $(f,s_{i})$-verified grids at a variety of thresholds. 
Depending on how many $(f, s_i)$-verified grids we are able to construct, the following corollary provides an \emph{a posteriori} bound on the bottleneck distance between our calculated persistence diagrams and the true persistent diagrams for our function.

\begin{corollary}
\label{prop::FinalResult}
%Fix a continuously differentiable function $f: [0,1]^N \to \R$. 
Fix a sequence of real numbers $\setof{s_i}_{i=-1}^{m+1}$ with the following properties:
\begin{enumerate}
\item $s_{-1} = -\infty$ and  $s_0 <  \min_{x\in X} f(x)$, 
\item $s_{m+1} = \infty$ and  $s_{m} > \max_{x\in X} f(x)$.
\end{enumerate}
Suppose there exists a sequence of grids $\{ \cG_{s_i} \}_{i=-1}^{m+1}$ such that each grid $\cG_{s_i}$ is $(f,s_i)$-verified. 
Define $\{\cE_{s_i} \}_{i=-1}^{m+1}$ to be the associated CW complexes, define  $\{ \bar{\cX}_{s_i} \}_{i=-1}^{m+1}$ to be the corresponding approximate filtration, and define $PD'$ to be the persistence diagram associated with  $\{ \bar{\cX}_{s_i} \}_{i=-1}^{m+1}$.  
Then 
\[
d_B ( PD(f) , PD') \leq \varepsilon 
\]
 where $\varepsilon  = \max_{1 \leq i \leq m} \abs{s_{i} - s_{i-1}} $. 
\end{corollary}
\begin{proof}

By Theorem \ref{thm::VerifiedMeansCompatible} it follows that each CW structure $\cE_{s_i} $ is compatible with $ X_{s_i}$. 
Moreover, for every cell $e\in \cE_{s_i}$ which satisfies Condition~(3) in Definition~\ref{def::CompatibleCW}, the function $f\circ h(\bx,s)$ is non-increasing in $s$. 
It follows from  Theorem \ref{prop::ApproximateFiltration}  that $\{ \bar{\cX}_{s_i} \}_{i=-1}^{m+1}$ and $\{ X_{s_i}\}_{i=-1}^{m+1}$ have the same persistence diagrams. 
For every $\delta >0 $ the sequence $\{ X_{s_i}\}_{i=-1}^{m+1}$ is an $(\varepsilon + \delta)$-approximate filtration. 
Taking the limit as $\delta$ goes to zero, it follows from Lemma \ref{prop::closeApproximation} that $d_B ( PD(f) , PD') \leq \varepsilon $.

\end{proof}

In practice, we do not know \emph{a priori} the thresholds for which we can construct an $(f,s_i)$-verified grid. 
Consequently, we can only give \emph{a posteriori} error estimates for our persistence diagrams.
For example, fix a sequence of real numbers $\{s_{i}\}_{i=-1}^{m+1}$ as above and suppose that it is uniformly spaced. 
That is  $ \Delta = | s_i - s_{i-1}|$ for $ 1 \leq i \leq m$.
If we attempt to create $(f,s_i)$-verified grids corresponding to $\{s_i\}_{i=-1}^{m+1}$ then our resulting persistence diagram will satisfy 
\begin{equation}
d_{B}( PD(f), PD')  \leq  \Delta ( F + 1)
\label{eq:PersistenceError}
\end{equation}
where $F$ is the largest number of consecutive thresholds $s_i$ which we failed to construct $(f,s_i)$-verified grids. 
While we have \emph{a priori} control on choosing $\Delta$, the final error estimate can only be determined after finishing the calculation.

The algorithm we employ for creating an $(f,t)$-verified grid will either terminate successfully or halt after reaching one of two failure conditions. 
The first failure condition is reached if the algorithm has exceeded a predetermined maximum resolution of the domain. 
This condition is imposed to prevent the algorithm from subdividing grid elements \emph{ad infinitum}.
The second failure condition occurs if, by using interval arithmetic, a vertex has an image which is neither strictly greater than nor strictly less than the threshold $t$. 
If this conditions is satisfied, then we are unable to decide whether or not to include that vertex in the approximation.

In general, we only need to worry about the first failure condition, and regions of the sub-level set which are geometrically complex. 
In particular, if a cube $C$ contains a critical point of the function $f$, then it can only be $(f,t)$-verified if it satisfies either $t > f(C)$ or $ t \leq f(C)$.  
Consequently, if $f$ has a critical value very close to the threshold $t$, then a dyadic grid will need to have a very fine resolution in order to be $(f,t)$-verified. 
Additionally, it is sometimes the case that the offspring of a cube which is $(f,t)$-verified are themselves not $(f,t)$-verified. 
For these reasons, having very precise interval arithmetic can help prevent runaway subdivisions of the domain.

\subsection{Computational Results}

We implemented our algorithm for studying the sub-level sets of functions defined on the unit square: $[0,1]^2$. 
The source code can be downloaded from  \cite{website:jaquette-kramar}.
In particular, we studied random Fourier series  of the form
\begin{eqnarray*}
f(x,y) &:=& \sum_{i=1}^N \sum_{j=1}^N 
a_{i,j,1} \sin( 2 \pi i x ) \sin( 2 \pi j y ) + 
a_{i,j,2} \sin( 2 \pi i x ) \cos( 2 \pi j y ) \\ && + 
a_{i,j,3} \cos( 2 \pi i x ) \sin( 2 \pi j y ) + 
a_{i,j,4} \cos( 2 \pi i x ) \cos( 2 \pi j y )
\end{eqnarray*}
where the coefficients $a_{i,j,k}$ were taken from a normal distribution of mean $0$ and standard deviation $1$.
Depicted in Figure \ref{fig::PersistenceDiagrams} is one such computation, where we used $N=5$  as the maximum number of modes.

\begin{figure}[tb]
\subfigure{\includegraphics[height = .31\textwidth]{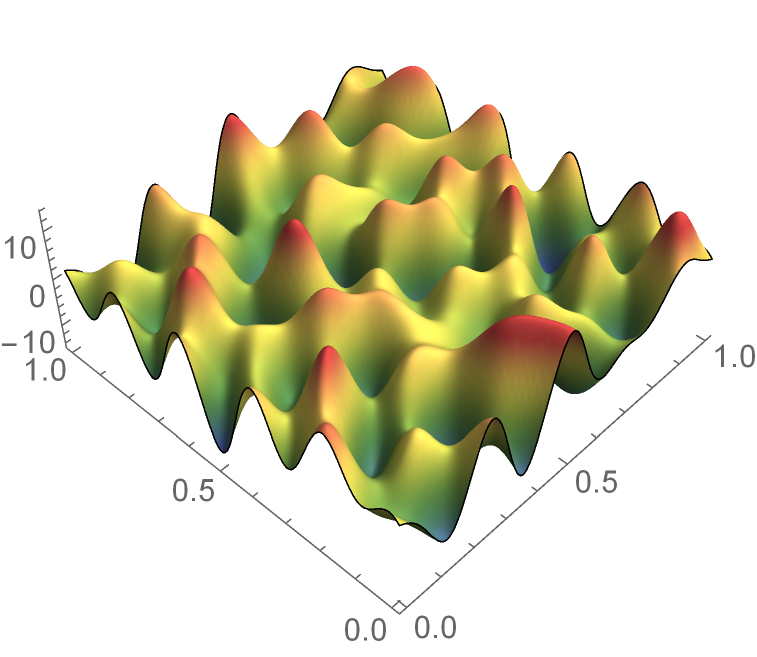}}
\subfigure{\includegraphics[height = .29\textwidth]{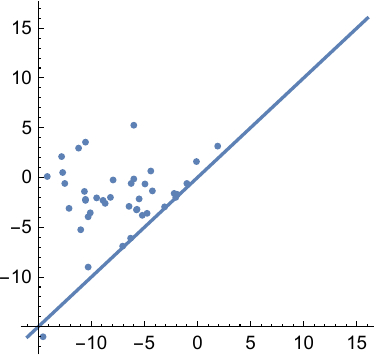}}
\subfigure{\includegraphics[height = .29\textwidth]{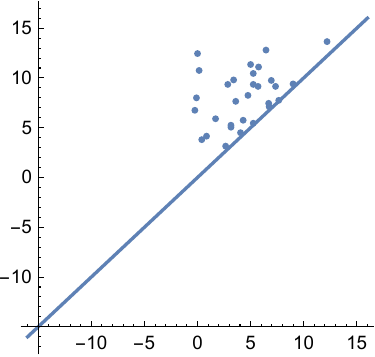}}
\caption{ (a) Plot of a trigonometric polynomial with random coefficients.   (b) $\beta_0$ persistence diagram and (c)  $\beta_1$ persistence diagram.}
\label{fig::PersistenceDiagrams}
\end{figure}

For the filtration we tried to construct $(f,s_i)$-verified grids for thresholds between $-15$ and $15$ using spacings of $0.05$.  
Out of the 601 thresholds, we failed to construct $(f,s_i)$-verified grids for 11 thresholds. 
The largest number of consecutive failures occurred when both $-5.2$ and $-5.15$ failed to be verified.  
Hence the largest gap between any two verified thresholds in our filtration was $0.15$. 
Therefore the persistent diagrams we calculated are within a bottleneck distance of  $0.15$  of the true persistent diagrams for our given function. 

The computation took 193 hours and used 16 MB of memory, and was performed on a single Intel i7-2600 processor.   
The algorithm created a filtered CW structure on $[0,1]^2$ which contained 145,721 cells, which was then processed by the \emph{Perseus} software program to compute its persistent homology \cite{mischaikow2013morse}. 
The $\epsilon$-approximate filtration was constructed sequentially by computing $\bar{\cX}_{s_i}:= \cX_{s_i} \cap \bar{\cX}_{s_{i+1}}$, so at most three CW complexes were stored in memory at a given time. 
At the cost of using more memory, one could reduce the computation time by computing the cellular approximations $\{ \cX_{s_i} \}$ in parallel, and subsequently creating the filtration $\{\bar{\cX}_{s_i}\}$.
Additionally, the one point of departure our implementation took from the algorithm here described is that the cells were represented using a lookup table, as opposed to the product of generalized intervals. 

The reason our algorithm failed to verify the $11$ thresholds is because when attempting to verify them, it surpassed 25 subdivisions of the domain. 
In general, the size of geometric features in these sub-level sets was much larger than $2^{-25}$. 
However, the error bounds produced by interval arithmetic have a risk of creating a cascade of cubes which fail to become $(f,t)$-verified.
In order to compute interval arithmetic more precisely, whenever we computed a function on a cube, we subdivided the cube seven times and computed the function using interval arithmetic on each piece.

When we repeated the computation but instead used six subdivisions, the number of thresholds we failed to verify increased $4.5$ fold and the accuracy of our computed persistence diagram was only $0.4$ in the bottleneck distance (it failed to verify seven consecutive thresholds). 
Using six subdivisions also increased the number of cells in the filtered complex by  $10.5\%$, however the entire computation finished $3.7$ times faster.
We refer the reader to \cite{CWD13randomized,DKW09verified} for more efficient techniques to compute precise interval arithmetic bounds needed by this algorithm.

\section*{Acknowledgments}
The authors would like to thank Konstantin Mischaikow for his insightful comments and conversations.

\bibliographystyle{amsplain}
\bibliography{BibDataBase}
\end{document}